\pdfoutput=1 
\documentclass[10pt, oneside]{article}
\newif\iflong
\longtrue
\usepackage{graphicx,color}
\usepackage{verbatim}
\usepackage{amssymb}
\usepackage{amsfonts}
\usepackage{amsmath,amsthm}
\usepackage{epsfig}
\usepackage{algorithm}
\usepackage{subfigure}
\usepackage[noend]{algorithmic}
\newtheorem{theorem}{Theorem}
\newtheorem{remark}{Remark}
\newtheorem{corollary}{Corollary}
\newtheorem{lemma}{Lemma}
\newtheorem{definition}{Definition}
\newtheorem{proposition}{Proposition}
\newtheorem{assumption}{Assumption}

\newcommand{\reals}{\ensuremath{\mathbb{R}}}
\newcommand{\naturals}{\ensuremath{\mathbb{N}}}
\newcommand{\D}{\displaystyle}

\renewcommand{\algorithmiccomment}[1]{\hfill //#1}

\usepackage[bookmarks=false,colorlinks=true,linkcolor=blue,citecolor=blue,urlcolor=blue]{hyperref}

\title{Sampling-Based Resolution-Complete Algorithms for Safety Falsification of Linear Systems}
\author {Amit Bhatia \thanks{Amit Bhatia is with the Department of Mechanical and Aerospace Engineering,
University of California at Los Angeles, Los Angeles, California
90095, {\tt\small abhatia@ucla.edu}} \and Emilio Frazzoli
\thanks{Emilio Frazzoli is with the Department of Aeronautics and
  Astronautics, Massachusetts Institute of Technology, Cambridge,
  Massachusetts, 02139,  {\tt\small frazzoli@mit.edu}}}
\date{}
\begin{document}
\maketitle
\pagestyle{empty}
\begin{abstract}
In this paper, we describe a novel approach for checking safety specifications of a dynamical system with exogenous inputs over infinite time horizon that is guaranteed to terminate in finite time with a conclusive answer. We introduce the notion of resolution completeness for analysis of safety falsification algorithms and propose sampling-based resolution-complete algorithms for safety falsification of linear time-invariant discrete time systems over infinite time horizon. The algorithms are based on deterministic incremental search procedures, exploring the reachable set for feasible counter examples to safety at increasing resolution levels of the input. Given a target resolution of inputs, the algorithms are guaranteed to terminate either with a reachable state that violates the safety specification, or prove that no input exists at the given resolution that violates the specification.
\end{abstract}
\section{Introduction}
\subsection{Background}
 Simulation-based techniques for formally verifying properties (called specifications) of discrete, continuous and hybrid systems, have come under great deal of attention recently. The motivation to use such techniques arises from the fact that most of the real-world systems are quite complex and operate in the presence of unknown external disturbances. As a result, verifying that each and every state of a system satisfies a given specification may be impractical, or in general even impossible. The problem of finding the set of all states the system can reach (called as the reachable set), based on its dynamics and initial conditions, is known as the {\em reachability problem} in literature. For continuous and hybrid systems, this problem is in general known to be undecidable~\cite{Kesten.Pnueli.ea:93,Alur.Courcoubetis.ea:95}. An important class of specifications are safety
specifications that describe the properties that the state of a system should satisfy, to be considered safe. For analyzing safety specifications of a system, a wide variety of methods have 
been proposed in~\cite{Henzinger.Ho:95,Alur.Dang.ea:02,Prajna.Jadbabaie:04,Chutinan.Krogh:99,Dang.Asarin.ea.:02,Kurzhanski.Varaiya:00,Mitchell.Tomlin:00,Girard:05,Ratschan.She:07}.

Most of these methods attempt to verify safety of a given system by over approximating the actual
reachable set. Hence, they are liable to generate a spurious counter
example which violates a specification, but is not a feasible
trajectory~\cite{Krogh.Silva.ea.:01}.
Even though refining the abstraction is usually possible, there is in
general no guarantee that the process of successive refinements would stop in
finite time~\cite{Krogh.Silva.ea.:01}. As a result, such methods can only verify safety of a system but will be inconclusive with regard to disproving it. Safety of a given system can be disproved only by working with either the actual reachable set, or, by giving a feasible counter example (constructed, e.g., using a simulation-based method). 
\subsection{Sampling-based algorithms for safety falsification}
\label{sec:mpd vs falsification}
To answer the complementary question of safety falsification, sampling-based incremental search algorithms have been proposed by us, and others, in~\cite{Bhatia.Frazzoli:04,Belta.Esposito.Kumar:05,Donze.Maler:07,Nahhal.Dang:07}, that are based on similar algorithms used in robotics~\cite{LaValle.Kuffner:01}. These algorithms try to falsify safety of the system quickly, but they are  only {\em probabilistically-complete} (see~\cite{LaValle.Kuffner:01}). This means that, the algorithms will find a counter example (if one exists) with probability 1, if they are
allowed to run forever. However, if they are terminated before a counter example is found, then they become inconclusive.\iflong One such instance is shown in Fig.~\ref{fig:pc vs rc intro}, at the left. The trajectories are generated using Rapidly Exploring Random Tree~\cite{LaValle.Kuffner:01}, a 
probabilistically-complete algorithm, for a point mass moving with bounded velocity in
two dimensions starting from a set containing origin. Terminating the search procedure before the
unsafe set is reached leads to the wrong
conclusion that the system is safe. Note that the samples are points ({\em zero volume sets}) and the sampling procedure is randomized.\fi
\iflong
\begin{figure}
\begin{minipage}[t]{0.5\textwidth}
\begin{center}
\includegraphics[scale=0.28,angle=0]{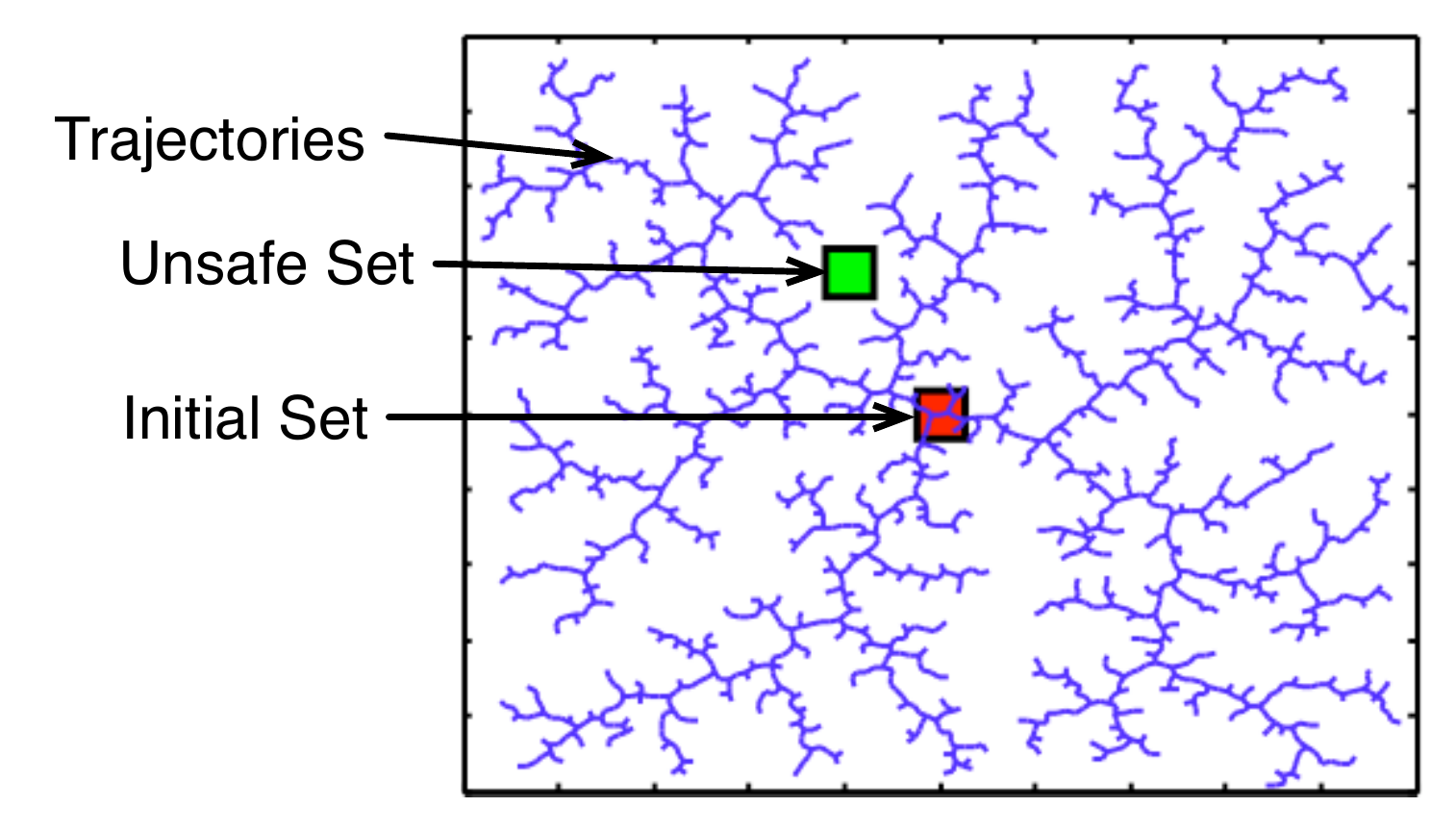}\\
\end{center}
\end{minipage}
\begin{minipage}[t]{0.5\textwidth}
\begin{center}
\includegraphics[scale=0.28,angle=0]{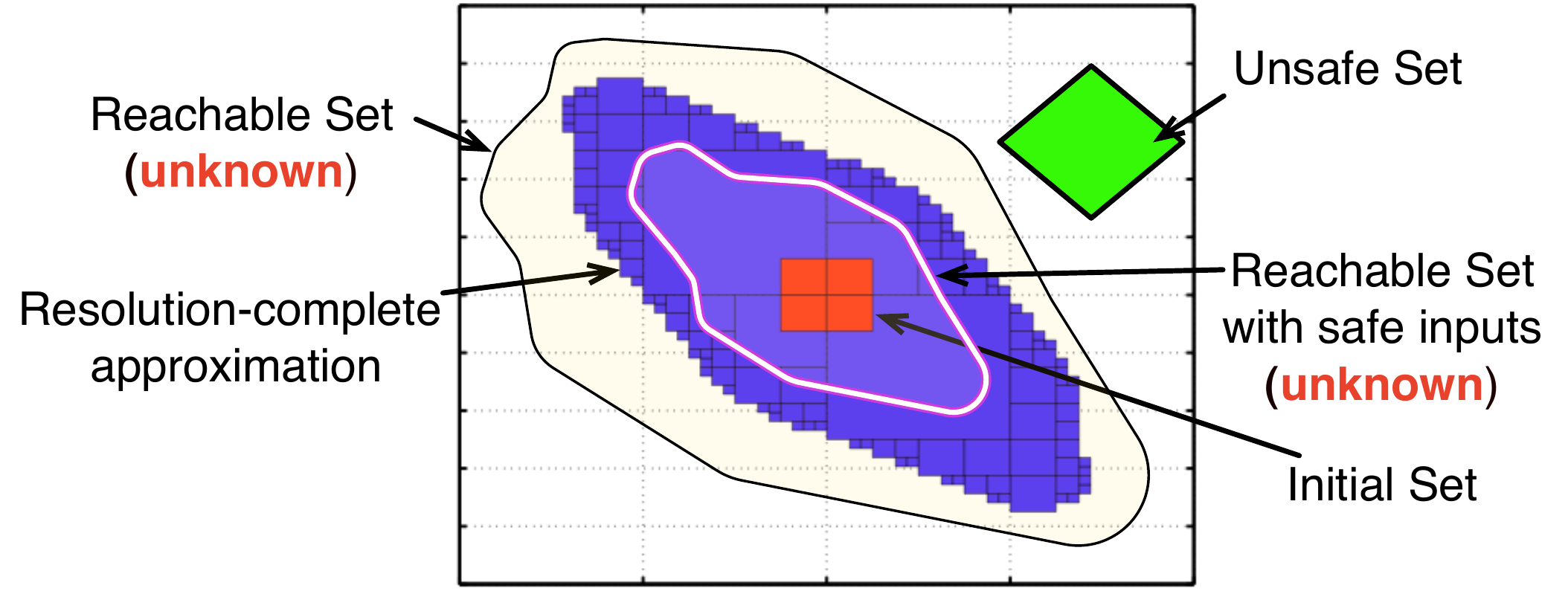}\\
\end{center}
\end{minipage}
\caption{Probabilistic vs resolution completeness}
\label{fig:pc vs rc intro}
\end{figure}
\fi

 To analyze the completeness properties
 of search based motion planning algorithms used in
 robotics, the notion of {\em resolution completeness} has been proposed 
 in~\cite{Cheng.LaValle:04b,Goldberg:95}. A resolution-complete algorithm, is guaranteed to find a solution (if one exists),
 in finite time, provided that, the resolution of discretization in
 input space, and state space, is high enough. 
In~\cite{Bhatia.Frazzoli:06,Bhatia.Frazzoli:07}, we introduced a similar notion for disproving 
safety of continuous and hybrid systems over infinite time horizon, and, proposed resolution-complete
 deterministic algorithms applicable to linear time invariant systems (abbreviated as LTI systems). The algorithms
 work by incrementally building trajectories in
 state space at increasing levels of resolution, such that, either they fetch a legitimate counter example, or a guarantee that no such counter example exists at
 given level of resolution (and hence a conclusive
 answer to the unsafety problem), in finite time. \iflong In Fig.~\ref{fig:pc vs rc intro}, at the right, we show an example where we end up with a non-zero volume under-approximation
to reachable set (when no counter example was found),
when using a sampling-based resolution-complete algorithm for safety falsification of a second order system with exogenous inputs~\cite{Bhatia.Frazzoli:06}.\fi Very recently, 
an alternate notion of resolution completeness has also been proposed by Cheng and Kumar for continuous-time systems for the case of finite horizon in~\cite{Cheng.Kumar:06}.
\subsection{Contributions of the paper and relation to other approaches}
 In this paper, we propose two new resolution-complete algorithms that use incremental grid-based sampling methods (similar to those proposed in~\cite{Lin.Yer.Lav:04}) with good coverage properties for exploring the state space. The first algorithm uses breadth-first-search based scheme with branch and bound strategy to explore the state space for counter examples to given safety specification, at increasing levels of resolution of the input. This algorithm can be applied to discrete-time LTI hybrid systems. Simulation results indicate that this algorithm, is an improvement over the one proposed by us in~\cite{Bhatia.Frazzoli:07} which is based on depth-first-search based scheme with branch and bound strategy. The second algorithm proposed in this paper can be used to falsify safety of discrete-time LTI continuous systems more efficiently, when the initial set is the equilibrium point. The reachable set for such initial conditions is a {\em convex} set. The proposed algorithm uses this fact to explore the state space more efficiently. \iflong Since both the algorithms are resolution-complete, they consider the safety falsification problem over infinite time horizon, with guarantees of finite-time termination of the search procedure and a conclusive answer at termination.
 
  An important feature that distinguishes our approach from alternate approaches recently proposed by others (e.g.~\cite{Girard:07}) for addressing safety over infinite time horizon, is that the requirements on discretization of state space in our case do {\em not} depend on time length of trajectories. This is an important advantage by itself, and also helps the algorithms in avoiding the so called {\em wrapping effect}~\cite{Kuhn:98}. This means that, in case no counter example is found, then the quality of the approximation constructed as a proof for safety of the system is not affected by time horizon. We also do {\em not} discretize the space of inputs to obtain completeness guarantees (unlike the approaches presented in~\cite{Cheng.Kumar:06,Girard:07}).
\else 
  Our notion of resolution completeness is defined on the space of exogenous control signals (unlike~\cite{Cheng.Kumar:06}); in our work, the requirements on discretization of state space, do {\em not} depend on time length of trajectories (unlike~\cite{Girard:07}); also, we do {\em not} discretize the space of inputs to obtain completeness guarantees (unlike~\cite{Cheng.Kumar:06,Girard:07}).
\fi

The paper is organized as follows. We introduce the framework for
describing hybrid systems and reachable sets in Section~\ref{sec:preliminaries} together with a formal definition of the notion of resolution completeness for safety falsification. In Section~\ref{sec:basic idea}, we explain the main idea used in the proposed algorithms for resolution completeness. 
 Conditions for resolution completeness of the proposed algorithms are discussed in Section~\ref{sec:derivations dth}. In Section~\ref{sec:sampling methods}, we explain the sampling method used by us in the algorithms and in Section~\ref{sec:rc algorithms}, we present the algorithms, together with the proof of their completeness. Simulation results are discussed in Section~\ref{sec:experiments}, and the paper is concluded in Section~\ref{sec:conclusions}.
\section{Preliminaries}\label{sec:preliminaries}
 In this section, we introduce notation for describing hybrid systems and reachable sets and define the notion of resolution completeness.
\begin{definition}[Hybrid System] We define a discrete-time LTI hybrid system
  $H$ as a tuple, $H= ( \mathcal{Q}, \mathcal{X}, \mathcal{U}, U, \Phi,
  \Delta, \mathcal{I}, \mathcal{S}, \mathcal{T}),$ where:
\begin{itemize}
\item $\mathcal{Q}$ is the discrete state space.
\item $\mathcal{X} \subseteq \reals^n$ is the continuous state space.
\item $\mathcal{U}$ is the family of allowed control functions
  equipped with a well defined metric. Each control function $u \in
  \mathcal{U}$ is a function $u:[0,t_f] \rightarrow U$, where $t_f \in \naturals$ is the terminal time of the trajectory.
The convex set $U \subset\reals^m$ is the input space. 
For simplicity, we assume $U$ to be a unit hypercube.
\item $\Phi: \mathcal{Q} \times \mathcal{X} \times \mathcal{U}
  \rightarrow \mathcal{X}$ is a function describing the evolution of the
  system on continuous space, governed by a difference equation of
  the form $x(i+1) = \Phi(x,q,u)=A_qx(i)+B_qu(i), i \in \naturals$, and $A_q,B_q$ are real matrices of size $n\times n, n\times m$ respectively.
  \item $\Delta \subset (\mathcal{Q} \times \mathcal {X})\times (\mathcal{Q} \times \mathcal {X})$, a relation
  describing discrete transitions in the hybrid states. Discrete
  transitions can occur on location-specific subsets $\mathcal{G}(q,q')
  \subseteq \mathcal{X}$, called guards, and result in jump relations
  of the form $(q,x) \mapsto (q', x)$.
\item $\mathcal{I}, \mathcal{S}, \mathcal{T} \subseteq \mathcal{Q}
  \times \mathcal{X}$ are, respectively, the invariant set, the
  initial set, and the unsafe set. 
\end{itemize}
\end{definition}
 The {\em semantics} of our model are defined as follows. When the
discrete state is in location $q$, the continuous state evolves
according to the difference equation $x(i+1) =A_q x(i)+B_q u(i), i\in \naturals$, for some value of
the input $u(i) \in U$, with $(q(0),x(0)) \in \mathcal{S}$.
In addition, whenever $x(i) \in \mathcal{G}(q(i),q')$, for some $q'$, the system
has the option to perform one of the discrete transitions modeled by
the relation $\Delta$, and be instantaneously reset to the new
discrete state $q'$, while the continuous state remains the same as before the discrete transition. The system
is required to respect the invariants by staying within $\mathcal{I}$
at all times. For the cases when the discrete state space is just a
single location, we
drop the discrete state $q$ from the notation.
 $\Omega = \{\bar{q}, \bar{q}: [0,t_f] \rightarrow Q\}$ denotes the set of trajectories on the discrete space. Trajectories of the system starting from $z \in\mathcal{Q} \times
 \mathcal{X} $ and using $u \in \mathcal{U}$, under the discrete
 evolution $\bar{q}$, are denoted by
 $\Psi(z,u,\bar{q}) \subset \mathcal{Q}\times \mathcal{X}$. A point on the trajectory $\Psi(z,u,\bar{q})$, reached at time $i
 \leq t_f$, is denoted by $\psi(z,u,\bar{q},i) \in\mathcal{Q}\times \mathcal{X}$. We will denote by $A^{\circ}$ the interior of a set $A$.
 \begin{definition}[Reachable Set] The reachable set $\mathcal{R}(\mathcal{U})$ for 
a system $H$ denotes the set of states that can be reached in the future. It is defined as
\begin{equation*}
\D \mathcal{R}(\mathcal{U})=\bigcup_{z \in
  \mathcal{S}, \;u \in \mathcal{U},\; \bar{q} \in \Omega} \Psi(z,u,\bar{q}).
\end{equation*}
\end{definition}
 We now formalize the notion of resolution completeness of an algorithm for safety falsification of a discrete-time LTI system with exogenous inputs. \footnote{For a similar notion applicable to more general class of systems, we refer the reader to~\cite{Bhatia.Frazzoli:07}.} 
\begin{definition}[Resolution completeness] A given algorithm is
  resolution-complete for safety falsification of a
system $H$, if there exists a sequence of
family of  control functions, $\{\mathcal{U}_{j} \}_{j=1}^{\infty}$, satisfying
 $\mathcal{U}_j\subset\mathcal{U}_{j+1}, \forall j$,
 and $\lim_{j \to \infty} \mathcal{U}_j = \mathcal{U}$ (in the sense of
 a given metric), such that, for any given $j \geq 1$, the algorithm terminates in finite time, producing, either a counter example $\psi(z_0,u,\bar{q},t)$, using a control
function $u \in \mathcal{U}, z_0 \in \mathcal{S}, \bar{q} \in \Omega, t\geq 0$, satisfying,
$\psi(z_0,u,\bar{q},t) \in \mathcal{T} $, or a guarantee that, $\mathcal{R}^{\circ}(\mathcal{U}_j) \cap\mathcal{T} = \emptyset$.
\end{definition}
\section{Basic Idea}
\label{sec:basic idea}
 One way to achieve completeness is to construct an approximation $\mathcal{R}_j$ (while searching for a counter example) that satisfies the set inclusion $\mathcal{R}^{\circ}(\mathcal{U}_j)\subseteq \mathcal{R}_j\subseteq\mathcal{R}(\mathcal{U})$
(shown in Fig.~\ref{fig:set inclusions}), thus guaranteeing feasibility of counter examples and safety 
  with respect to control functions belonging to $\mathcal{U}_j$. The algorithms that we propose in this paper use this idea.  
 To construct $\mathcal{R}_j$, the algorithms discretize the state-space using {\em multi-resolution} grids. For a discrete location $q \in \mathcal{Q}$, $G(q)$ denotes the multi-resolution grid for location $q$ that over-approximates $\mathcal{R}\cap\mathcal{I}(q,\cdot)$.
   The algorithms keep a record of the portion of $G(q)$ that is found to be reachable (denoted as $G_f(q)$), either {\em a priori} or during an execution of the algorithm. $G_u(q)$ denotes the rest of $G(q)$, i.e. $G_u(q)=G(q)\setminus G_f(q)$. The algorithms progressively sample regions of $G_u(q)$ (called {\em cells}) at increasing levels of resolution and try to construct trajectories that end in the sampled cell starting from somewhere in $G_f(q)$. The conditions for state-space discretization derived in Section~\ref{sec:derivations dth} guarantee that finding one feasible point  $\psi(z_0,u,\bar{q},t)$ in a cell $\xi(\varepsilon_j(q))$ of size $\varepsilon_j(q)$, with
$z_0 \in G_f(q), u \in \mathcal{U}_j, \bar{q} \in \Omega$, is enough to claim that $\xi \subset \mathcal{R}(\mathcal{U})$.
\iflong
\begin{figure}
\centering
\subfigure[\scriptsize Initialization step]{
\includegraphics[scale=0.25,angle=0]{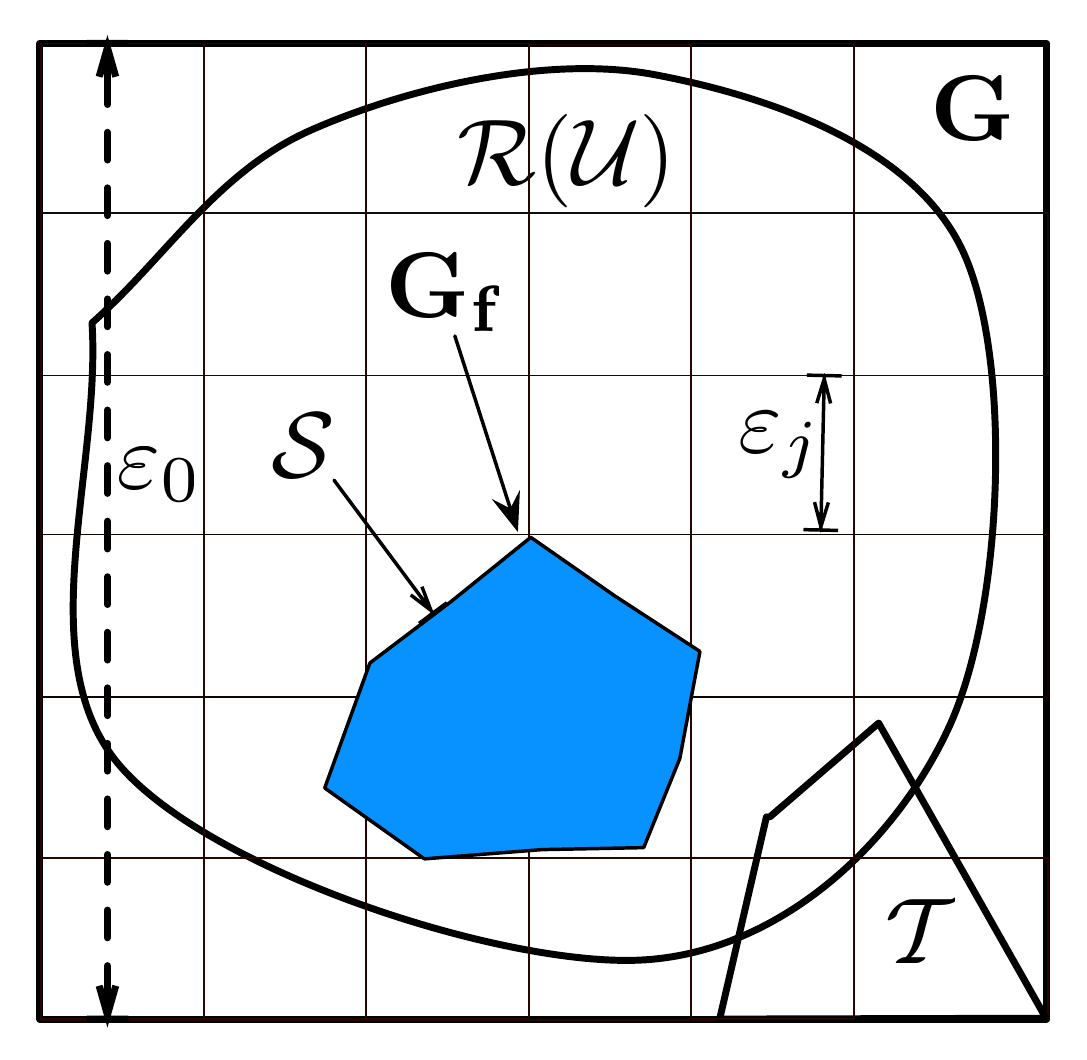}}\hfill
\subfigure[\scriptsize Exploration step]{
\includegraphics[scale=0.25,angle=0]{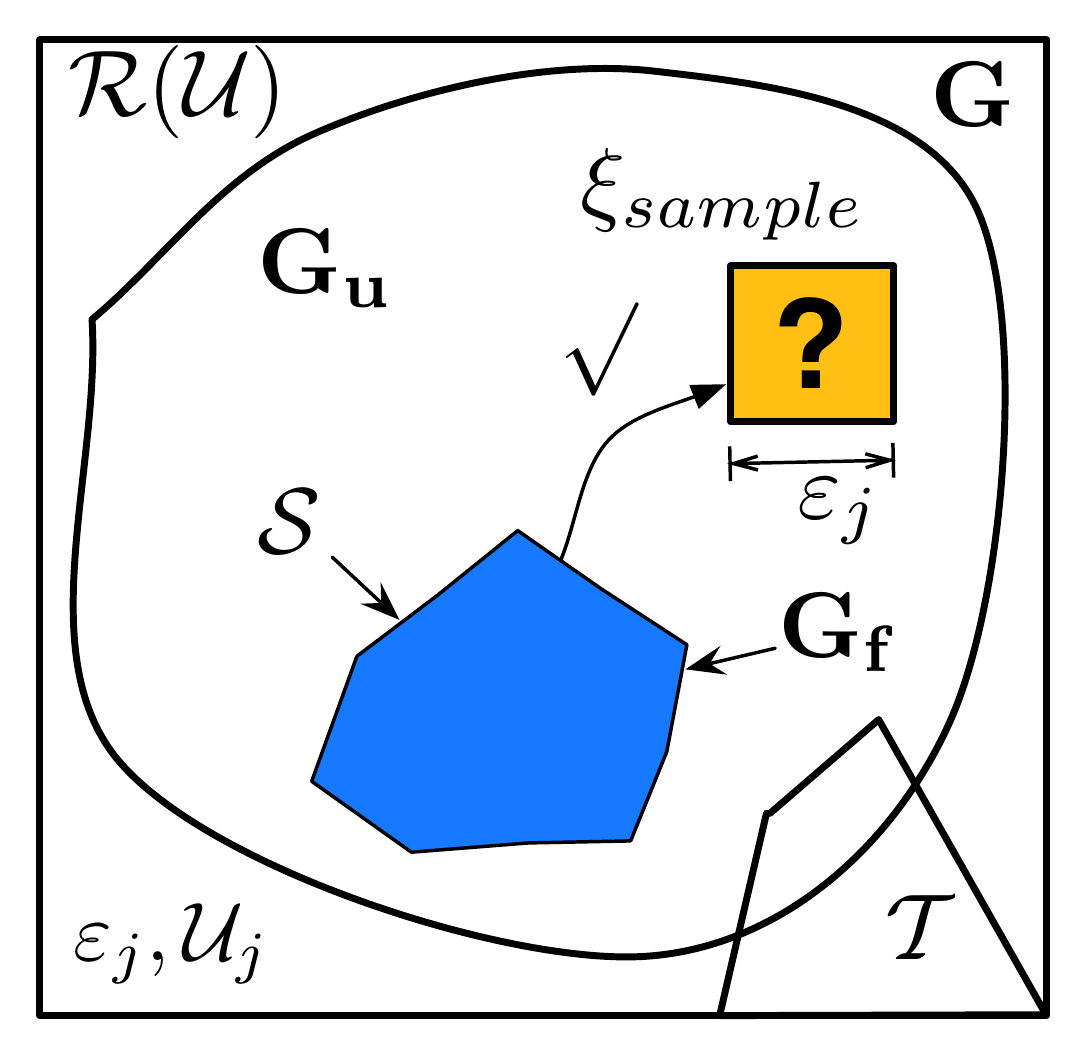}}\hfill
\subfigure[\scriptsize Refinement step]{
\includegraphics[scale=0.25,angle=0]{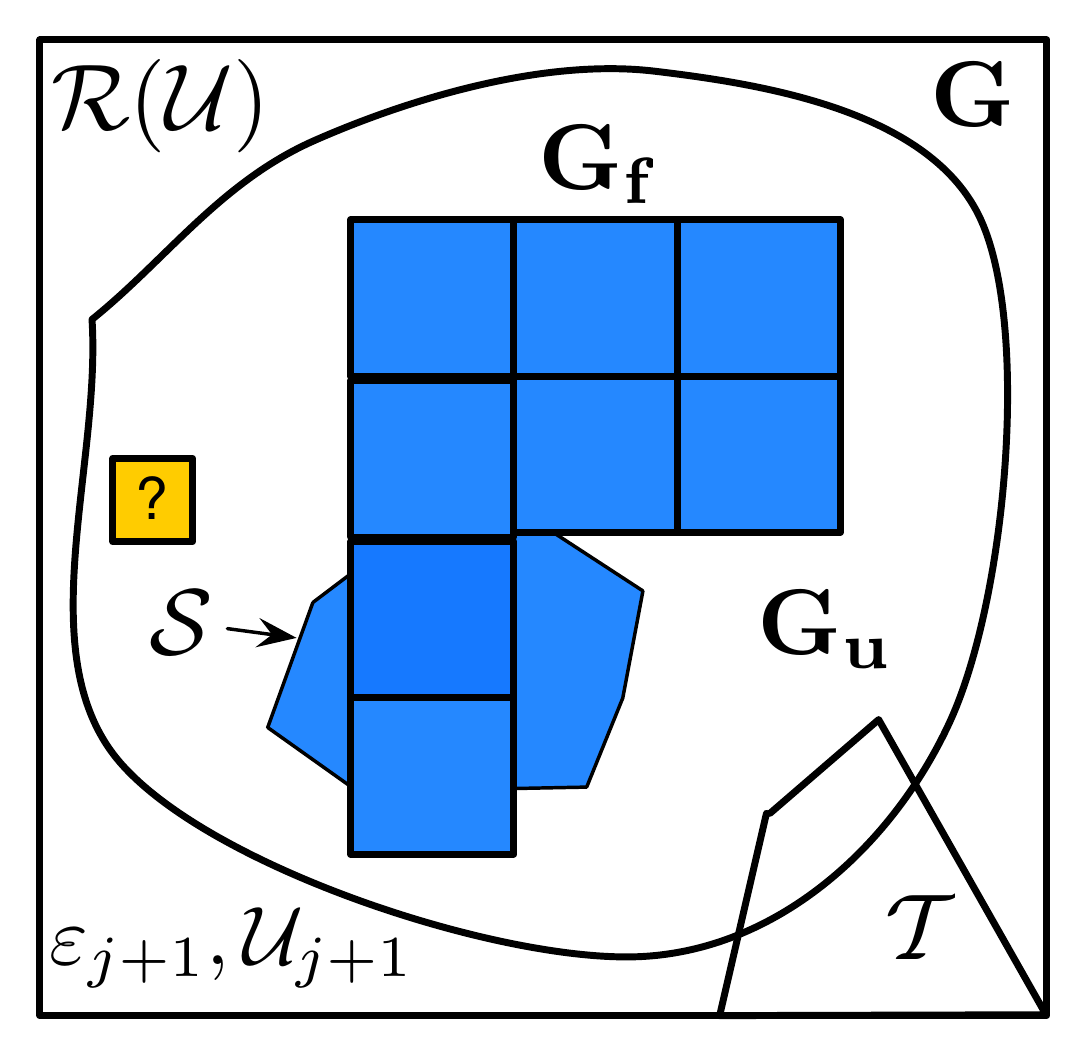}}\hfill
\subfigure[\scriptsize Termination step]{
\includegraphics[scale=0.25,angle=0]{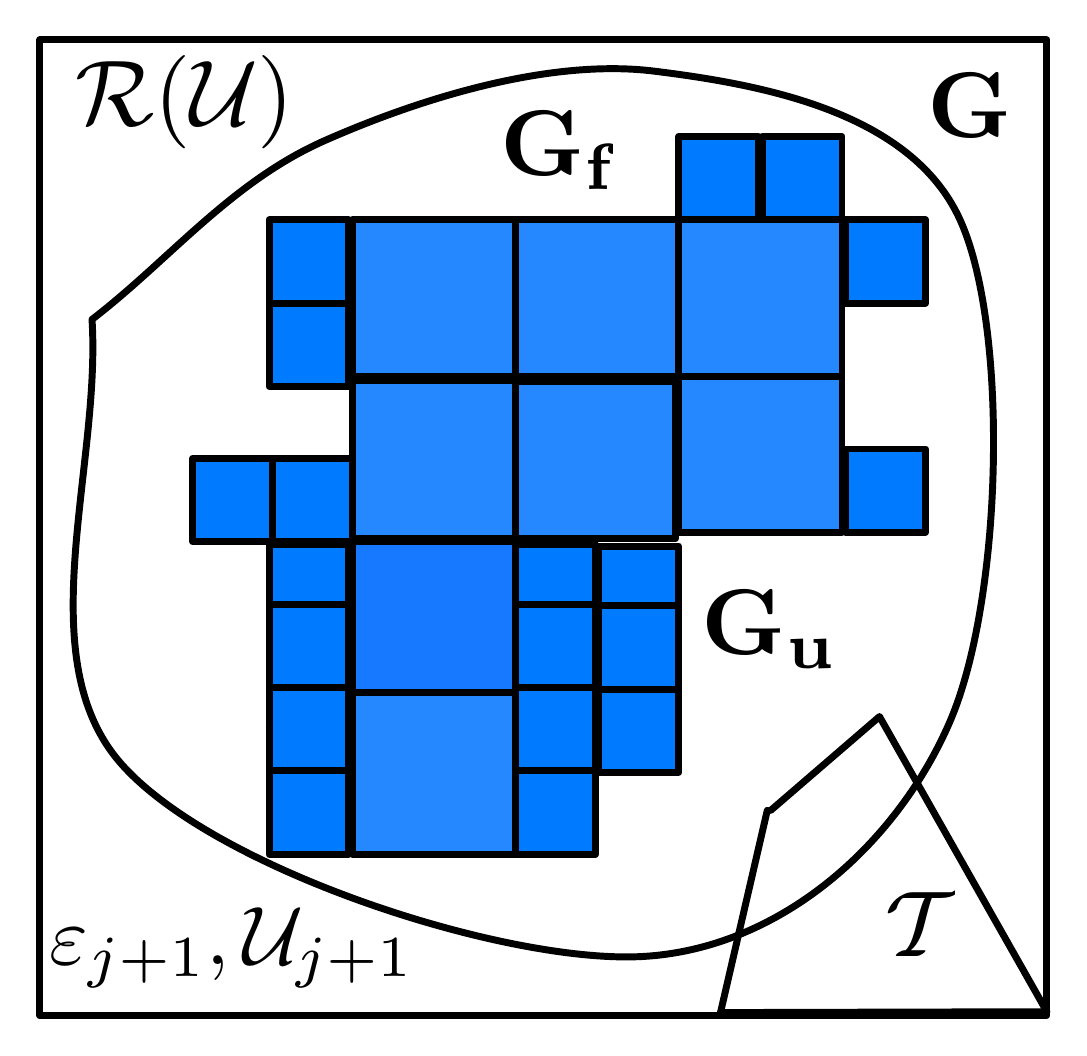}}\hfill
\caption{Execution of a resolution-complete algorithm for $\operatorname{card}{(\mathcal{Q})}=1$}
\label{fig:alg basic idea}
\end{figure}
In Fig.~\ref{fig:alg basic idea}, we show an execution of such a resolution-complete safety falsification algorithm (starting at resolution $j$ and stopping at $j+1$) for the case when $\operatorname{card}{\mathcal{Q}}=1$.
 
 We would like to remark here that, if we can find conditions that ensure the set inclusion, $\mathcal{R}^{\circ}(\mathcal{U}_j)\subseteq \mathcal{R}_j\subseteq\mathcal{R}(\mathcal{U})$ for nonlinear systems, then similar algorithms can be used for resolution complete safety falsification of non linear systems as well.
\else
\begin{figure}
\centering
\subfigure[\scriptsize Initialization step]{
\includegraphics[scale=0.25,angle=0]{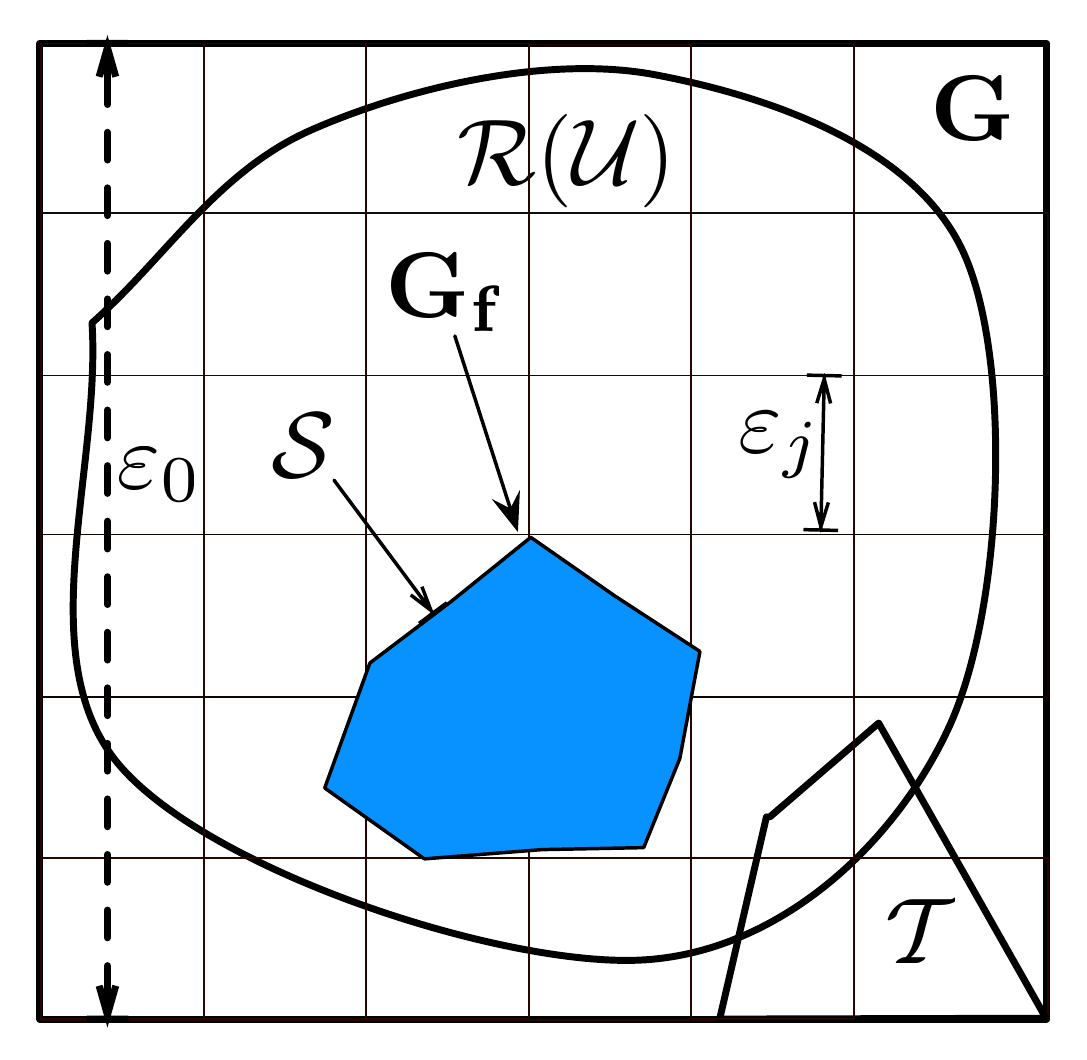}}\hfill
\subfigure[\scriptsize Exploration step]{
\includegraphics[scale=0.25,angle=0]{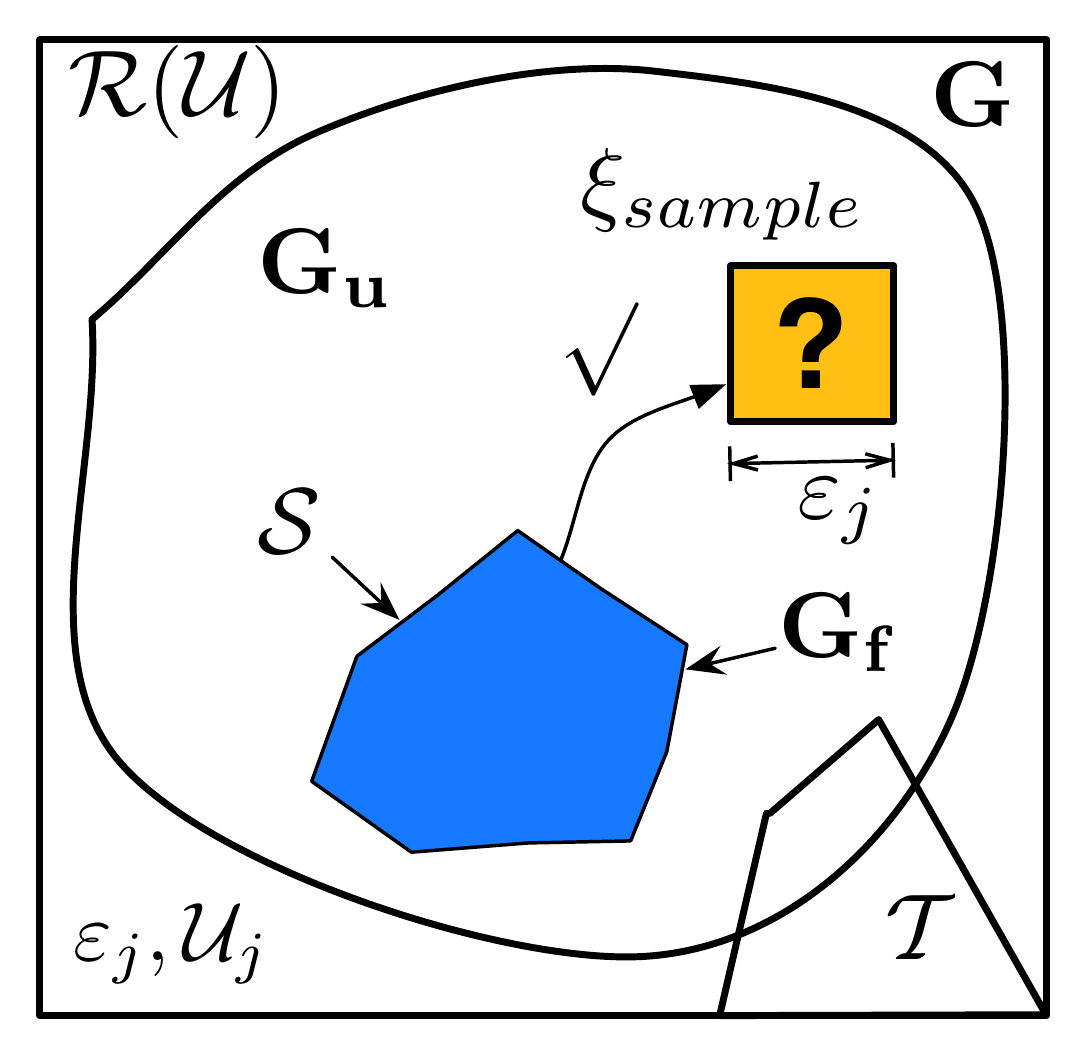}}\hfill
\subfigure[\scriptsize Refinement step]{
\includegraphics[scale=0.25,angle=0]{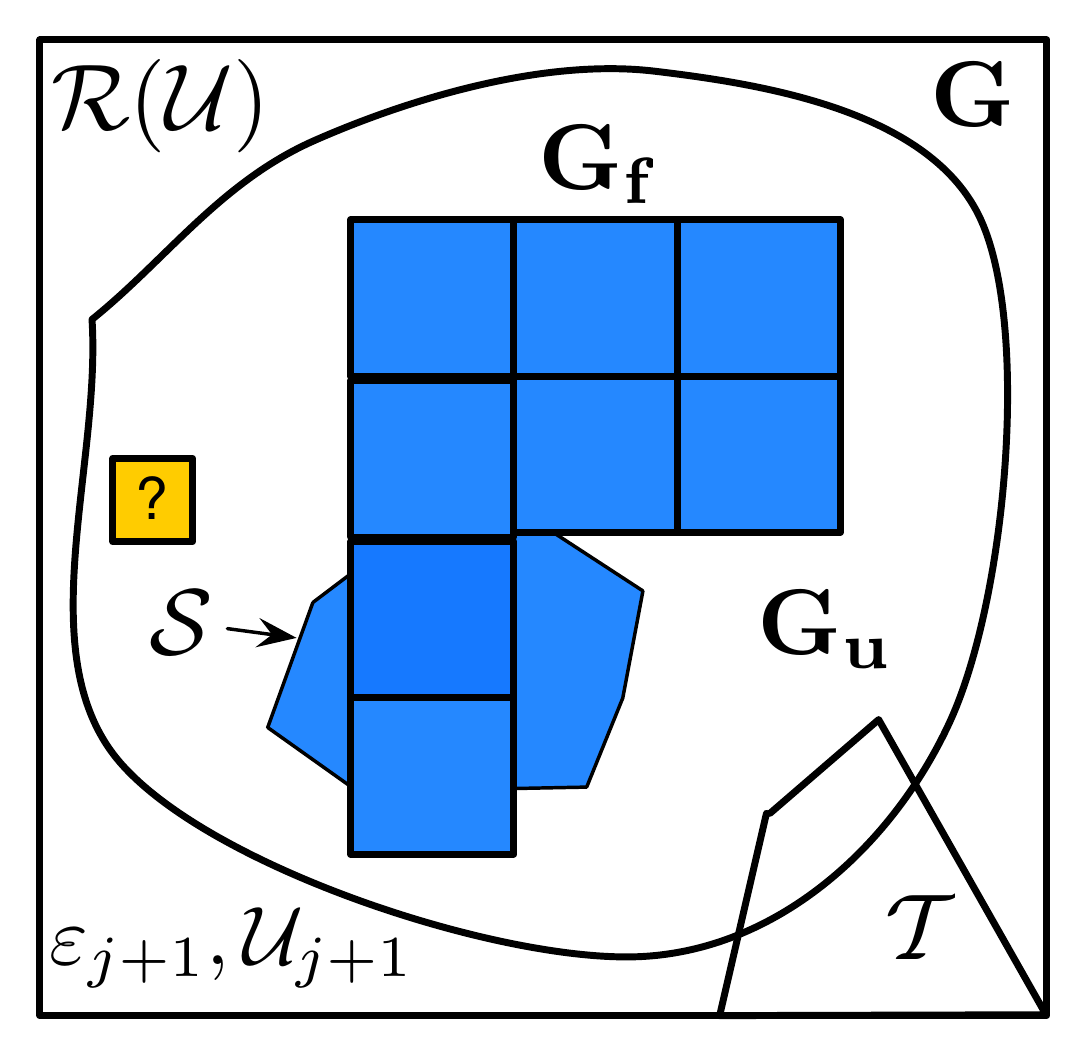}}\hfill
\subfigure[\scriptsize Termination step]{
\includegraphics[scale=0.25,angle=0]{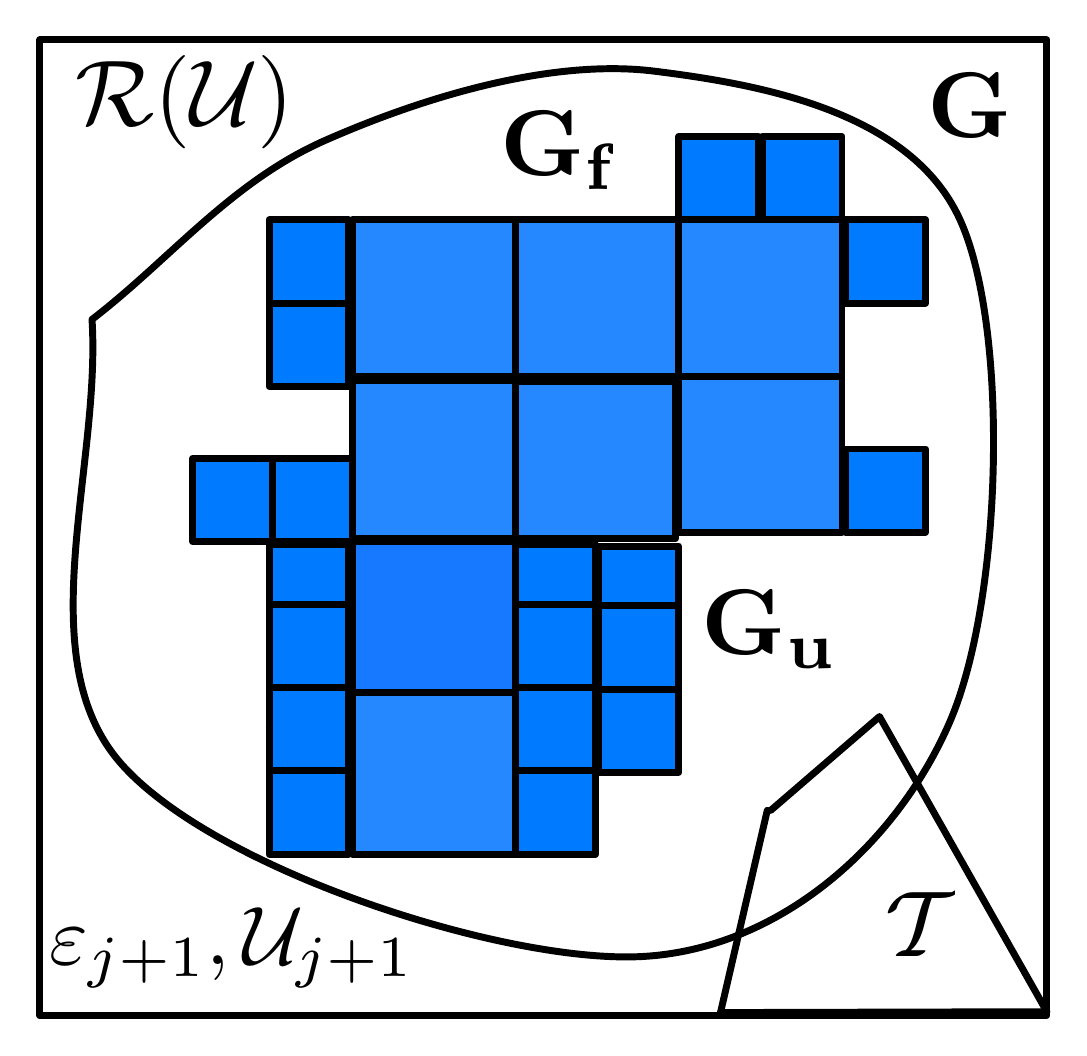}}\hfill
\caption{Execution of a resolution-complete algorithm for $\operatorname{card}{(\mathcal{Q})}=1$}
\label{fig:alg basic idea}
\end{figure}
In Fig.~\ref{fig:alg basic idea}, we show an execution of such a resolution-complete safety falsification algorithm for the case when $\operatorname{card}{(\mathcal{Q})}=1$.
\fi
\section{Conditions for Resolution-Complete Safety Falsification}
\label{sec:derivations dth}
  In this section, we derive necessary conditions for resolution-complete safety falsification of discrete-time LTI hybrid systems. For all the discussion that follows, $\lVert \cdot
\rVert$ denotes the infinity norm. The set of control functions is, $\mathcal{U}=\{u(\cdot):\lVert u(i) \rVert \leq 1,\;
 \forall i\in \naturals, i\leq t_f \} $.
$\mathcal{C}_q = [B_q\; A_qB_q \ldots A_q^{n-1}B_q]$
denotes the controllability matrix of the system in location $q \in
\mathcal{Q}$.
\subsection{Control functions for resolution completeness}
\label{subsec:control functions cont}
For resolution completeness\footnote{Resolution is defined on the space of control functions, $\mathcal{U}$.}, we need a sequence of control
functions $\{\mathcal{U}_j\}_{j=1}^{\infty}$ such that 
$\mathcal{U}_j\subset\mathcal{U}_{j+1}, \forall j \in \naturals$
 and $\lim_{j \to \infty} \mathcal{U}_j = \mathcal{U}$ in some metric
 defined on the space of control functions. We consider sequence of
 families of piece-wise constant control functions for our algorithm.
\begin{proposition}{\label{prop:control approx cont}}For given family of
 control functions $\mathcal{U}$, the sequence
  of family of control functions $\{\mathcal{U}_j \}_{j=1}^{\infty}$, 
$\mathcal{U}_j=\{u(\cdot):\lVert u(i) \rVert \leq l_j,\;
 \forall i \in \naturals, i \leq t_f, t_f \in \naturals \} $, with
 $\{l_j\}$ a strictly non decreasing sequence of real numbers and $\lim_{j
   \to \infty} l_j = 1$ satisfies 
$\mathcal{U}_j\subset\mathcal{U}_{j+1}, \forall j$ and 
$\lim_{j \to \infty} \mathcal{U}_j =\mathcal{U} $ in $L_{\infty}$ norm.
\end{proposition}
\begin{proof}
 The proposition is proved by stated requirements on $\{l_j\}$.  
\end{proof}
\subsection{Assumptions}
\label{subsec:system assumptions cont}
\begin{assumption}
\label{assmptn: ctb stb cont} For each discrete 
location $q \in \mathcal{Q}$, the system $H$ is stable at the origin, $\operatorname{rank}{(B_q)}=m$, and, $\operatorname{rank}{(\mathcal{C}_q)}=n$.
\end{assumption}
 Stability  (along with {\em Assumption}~\ref{assmptn:lin eq cont}) guarantees that $G(q)$ has a finite volume. $\operatorname{rank}{(B_q)}=m$ and $\operatorname{rank}{(\mathcal{C}_q)}=n$ is needed to be able to use Propositions~\ref{prop:nonemptyint cont},~\ref{prop:cell size
  cont}.
\begin{assumption}
\label{assmptn:lin eq cont}
  For each discrete location $q \in \mathcal{Q}$, $\mathcal{S}(q,\cdot),\mathcal{G}(q,\cdot)$ are specified as convex polytopes and $\mathcal{I}(q,\cdot), \mathcal{T}(q,\cdot)$ are specified as a convex polyhedra.
\end{assumption}
 At each step, the algorithms incrementally build the trajectories and check for unsafety and discrete mode switches by
  solving linear programs. Hence we need the sets
  $\mathcal{S}(q\cdot),\mathcal{I}(q,\cdot), \mathcal{G}(q,\cdot), \mathcal{T}(q,\cdot)$ to be convex polyhedra. Boundedness of $\mathcal{S}(q,\cdot),\mathcal{G}(q,\cdot)$ is used to prove finite time termination of the algorithms.

We now consider continuous dynamics in a given discrete location
 $q\in \mathcal{Q}$, and derive sufficient conditions for
 discretization of cells in $G(q)$.
\iflong
\begin{figure}
\centering
\subfigure[\scriptsize Basic idea]{
\label{fig:set inclusions}
\includegraphics[scale=0.25,angle=0]{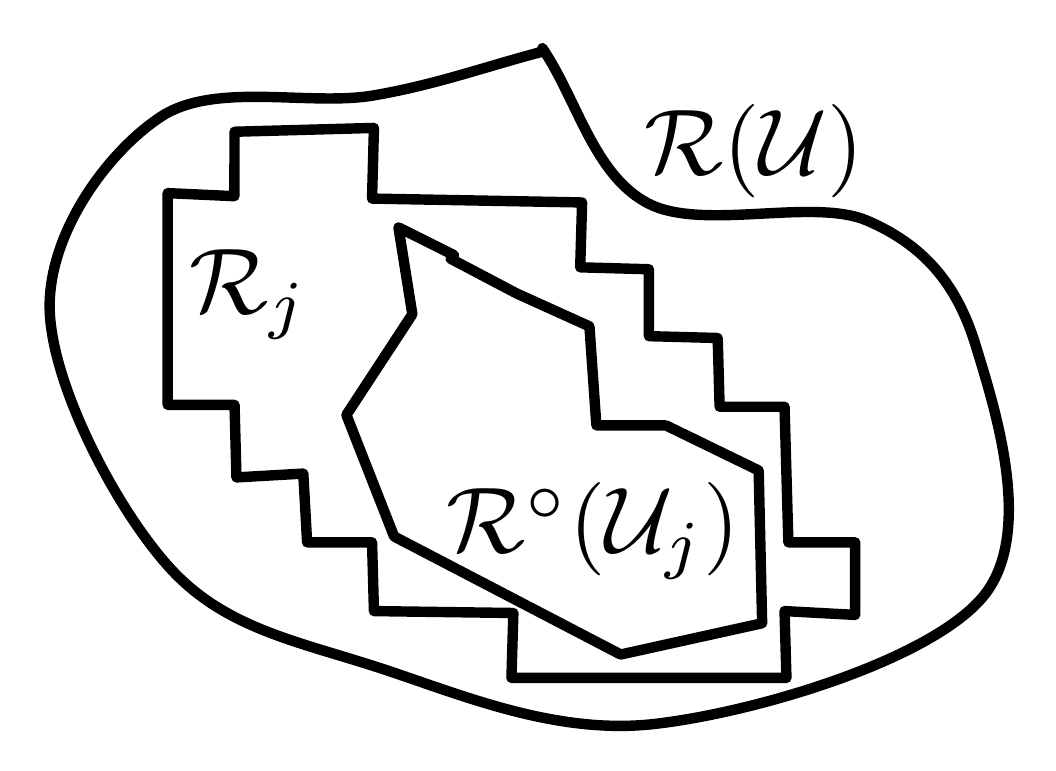}}\hfill
\subfigure[\scriptsize One-step reachability]{
\label{fig:onestep_completeness}
\includegraphics[scale=0.25,angle=0]{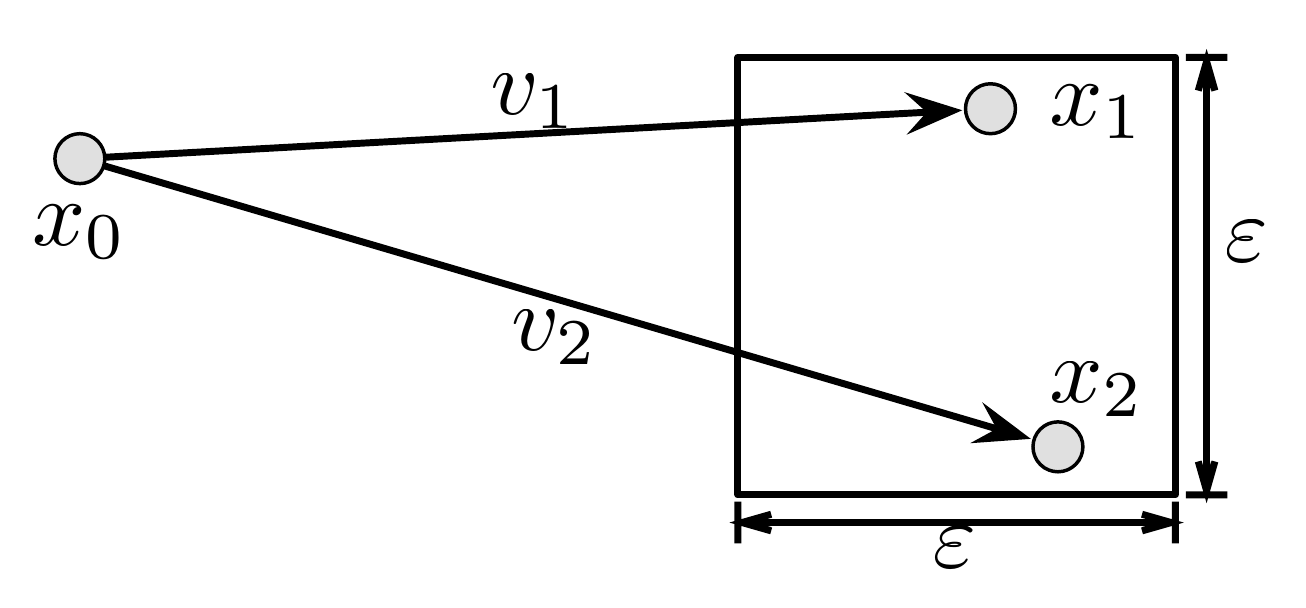}}\hfill
\subfigure[\scriptsize k-step reachability]{
\label{fig:kstep_completeness}
\includegraphics[scale=0.25,angle=0]{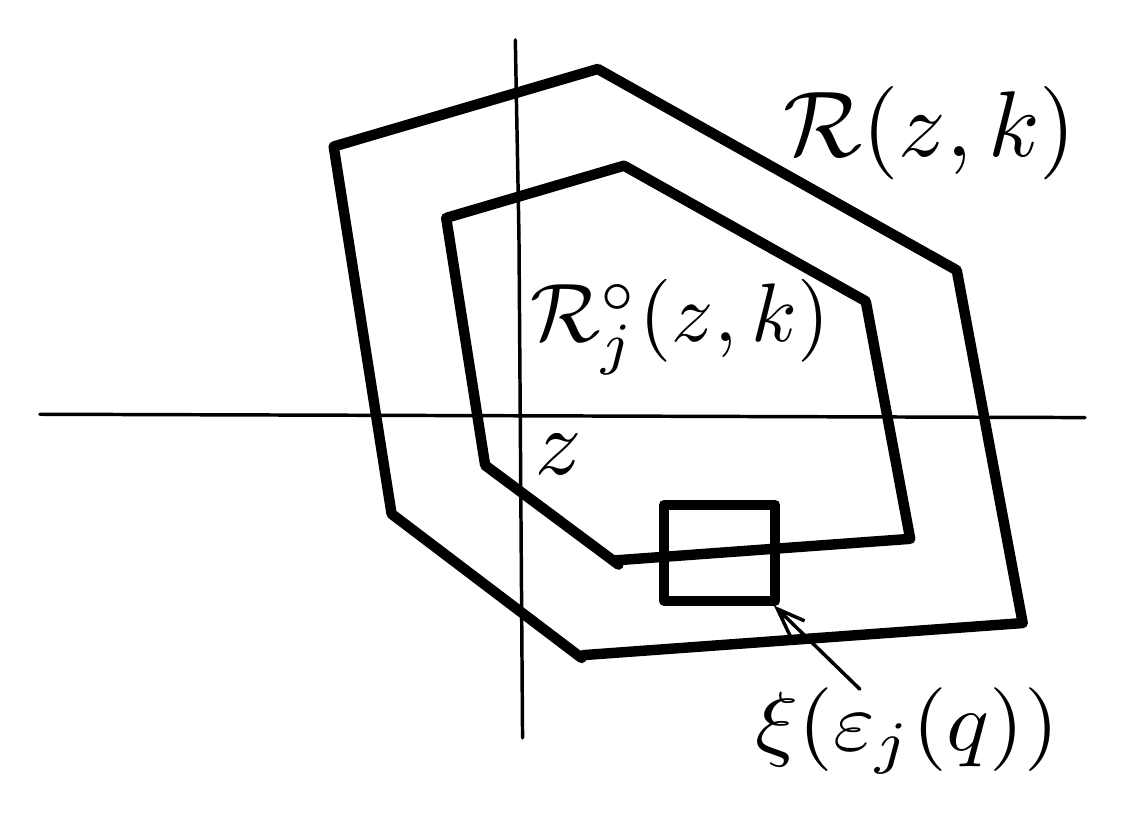}}\hfill
\caption{Set inclusions for resolution completeness}
\label{fig:overunder-onestep_inscribe}
\end{figure}
\else \vspace{-0.2 in}
\begin{figure}
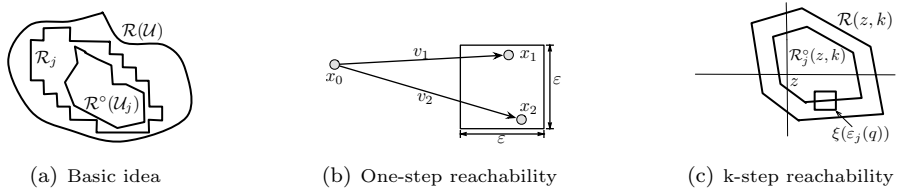

\centering
\subfigure[\scriptsize Basic idea]{
\label{fig:set inclusions}
\includegraphics[scale=0.25,angle=0]{figures/overunder-rotate}}\hfill
\subfigure[\scriptsize One-step reachability]{
\label{fig:onestep_completeness}
\includegraphics[scale=0.25,angle=0]{figures/2pt_completeness}}\hfill
\subfigure[\scriptsize k-step reachability]{
\label{fig:kstep_completeness}
\includegraphics[scale=0.25,angle=0]{figures/onestep_inscribe}}\hfill
\caption{Set inclusions for resolution completeness}
\label{fig:overunder-onestep_inscribe}
\end{figure}
\vspace{-0.4in} 
\fi
\subsection{Sufficient conditions for state space discretization}
\label{subsec:ss discretization cont}
 We first state an important proposition that guarantees that the set of
 points $\psi(z,u,\bar{q},k)$, that
 can be searched by the algorithms for feasibility, starting from $z$,
 with $u \in \mathcal{U}_j$ (as in Proposition~\ref{prop:control approx cont}) in
 $k\in \naturals$ steps has a non zero volume. 
 \iflong This follows from the
 assumption that $\operatorname{rank}{(B_q)}=m$, 
 and, $\operatorname{rank}{(\mathcal{C}_q)}=n$ in each
 discrete location $q \in \mathcal{Q}$.\fi
\begin{proposition}{\label{prop:nonemptyint cont}} For a given system
  $H$ satisfying {\em Assumption}~\ref{assmptn: ctb stb cont}, and for a
  given location $q \in \mathcal{Q}$, the set of points
  reachable by using a control function $u \in
  \mathcal{U}_j$ (as in Proposition~\ref{prop:control approx cont}) over $k$ steps has a
 non empty interior if $\lceil n/m \rceil \leq k \leq n$ for any $j
 \geq 1$, where $n$ is the dimension of the continuous state space and
 $m$ is the dimension of input space.
\end{proposition}
\begin{proof}
 Let $k$ be chosen such that $\lceil n/m \rceil \leq k \leq n$. The
 continuous dynamics of the system are  $x(i+1) = A_q x(i) + B_q u$, 
where $u \in \mathcal{U}_j$. Applying this over k steps, we get $x(k+i)= A_q^k x(i) + [B_q \; A_qB_q \ldots A_q^{k-1}B_q]\tilde u$, where $\tilde u \in \reals^{km}$ is the augmented
 input over $k$ steps.
Let $v =[B_q\; A_qB_q \ldots A_q^{k-1}B_q]\tilde
u$. Then the dynamics can be written as $x(k+i)= A_q^kx(i)
+v, v \in \reals^n$. Since $\operatorname{rank}{(\mathcal{C}_q)}=n$,
$\operatorname{rank}{[B_q\; A_qB_q\ldots A_q^{k-1}B_q]}=n$. Hence the set of
points reachable in $k$ steps is guaranteed to have a non empty interior.  
\end{proof}
 We now derive a conservative upper bound on the discretization of $G(q)$, so that
 for a cell $\xi$, if $\xi \cap \mathcal{R}^{\circ}(\mathcal{U}_j) \cap \mathcal{I}(q,\cdot)\neq \emptyset$ then $\xi
   \subset \mathcal{R}(\mathcal{U})\cap \mathcal{I}(q,\cdot)$. To do so, we first prove the result for a simpler case in Lemma~\ref{lemma:discretization}, and then prove the main result in Propositon~\ref{prop:cell size cont}.
\begin{lemma}
\label{lemma:discretization}
 Consider a continuous system $H$, with 
$x(i+1)=Ax(i)+v$ with, $x$, $v \in \reals^n$. For a given $x_1$, and an $\alpha_1>0$, with 
$x_1=Ax_0+v_1, x_0, v_1\in \reals^n, and, \lVert v_1\rVert \leq \alpha_1$, the following holds true:
For any $x_2 \in \reals^n$, and a given $\alpha_2>\alpha_1$, if $\lVert x_1-x_2\rVert \leq  \varepsilon$, and $\varepsilon \leq \alpha_2-\alpha_1$, then $\exists v_2 \in \reals^n$, such that, $x_2 = Ax_0+v_2$, with, $\lVert v_2 \rVert \leq \alpha_2$.
\end{lemma}
\begin{proof}
 Please refer to Fig.~\ref{fig:onestep_completeness}. $v_2 = v_1+x_2-x_1$ proves the result.  
\end{proof}
 \begin{proposition}{\label{prop:cell size cont}}
 Consider a system $H$ satisfying {\em Assumption}~\ref{assmptn: ctb stb
   cont}, with sequence of families of control
 functions $\{\mathcal{U}_j \}_{j=1}^{\infty}$, as in
 Proposition~\ref{prop:nonemptyint cont}. Let $j \in \naturals$ and $q\in \mathcal{Q}$ be fixed. 
Then, if the cell size $\varepsilon_j(q)$
 for a cell $\xi \in G(q)$ satisfies the bound,
$\varepsilon_j(q) \leq (1-l_j) /\lVert {\Gamma}^{+}_q \rVert$, 
where ${\Gamma_q} = [B_q\; A_qB_q \ldots A_q^{k-1}B_q]$ and
${\Gamma_q}^{+}$ is the pseudo inverse,
 then $\xi \cap \mathcal{R}^{\circ}(\mathcal{U}_j) \cap \mathcal{I}(q,\cdot)\neq \emptyset \Rightarrow 
\xi \subset \mathcal{R}(\mathcal{U})\cap \mathcal{I}(q,\cdot)$.
\end{proposition}
\begin{proof}
The dynamics of the system over $k$ steps can be written as $x(k+i)= A_q^kx(i) + v$, with 
$v=\Gamma_q \tilde u$. ${\Gamma_q} = [B_q\; A_qB_q \ldots A_q^{k-1}B_q]$ and 
   $\tilde u \in \reals^{km}, \lVert \tilde u \rVert \leq l$ is the augmented input over $k$ steps. This implies that $\tilde u = \Gamma_q^{+}v$, where $\Gamma_q^+$ is the pseudo inverse of $\Gamma_q$.  Finding the tightest bounds on $v$ is hard. However a conservative bound on $v$ is $\lVert v \rVert \leq l/\Gamma^+$.
 Now, let $v_1= \Gamma_q \tilde u_1$ and $v_2= \Gamma_q \tilde u_2$ with $\tilde u_1,\tilde u_2 \in \reals^{km}$ and $\lVert \tilde u_1\rVert \leq l_j $ and $\lVert \tilde u_2 \rVert \leq 1$. 
 Let $\mathcal{R}(z,k)$ denote set of points reachable by the system by
 using  $u \in \mathcal{U}$ with $t_f = k$, and 
 $\mathcal{R}_j^{\circ}(z,k)$ the interior of the set of points reachable by the system by
 using  $u' \in \mathcal{U}_j$ with $t_f=k$, under continuous evolution, starting
 from $z$. This is shown in Fig.~\ref{fig:kstep_completeness}.
  The result of Lemma~\ref{lemma:discretization} implies that for $\varepsilon_q \leq (1-l_j)/\Gamma_q^+$,
  $\xi \cap \mathcal{R}_j^{\circ}(z,k) \neq \emptyset \Rightarrow 
\xi \subset \mathcal{R}(z,k)$. This proves that $\xi \cap \mathcal{R}^{\circ}(\mathcal{U}_j) \cap \mathcal{I}(q,\cdot)\neq \emptyset \Rightarrow 
\xi \subset \mathcal{R}(\mathcal{U})\cap \mathcal{I}(q,\cdot)$.  
\end{proof}

\section {Incremental Grid Sampling Methods}
\label{sec:sampling methods}
 As we discussed in Section~\ref{sec:basic idea}, for each location $q \in \mathcal{Q}$, we use a multi-resolution grid. Assume for this section that, $\operatorname{card}{(\mathcal{Q})}=1$, and that, $G$ is a $n$-dimensional unit cube, whose origin is the origin of coordinate axis. Using an iteratively refined multi-resolution grid ensures that for a given $j$, $G_u$ with resolution $\varepsilon_{j}$ does not have to be built from scratch. 
   For our work, we will use multi resolution {\em classical grids}. A multi-resolution classical grid at resolution level $r$ has $2^{rn}$ points. Moreover, it contains all the points of all the resolution levels $r' < r$. Every grid point $P_i$, in a multi resolution classical grid at resolution level $r$, can be written as 
   $P_i=\{\frac{a_1}{2^r},\ldots, \frac{a_n}{2^r}\}$ with $0 \leq a_1, \ldots, a_n \leq 2^r-1$,  $a_1,\ldots,a_n \in \naturals$, with the corresponding grid region (that we call as {\em cells}) $\xi(r,i)=[\frac{a_1,a_1+1}{2^r})\times \ldots [\frac{a_n,a_n+1}{2^r})$. Here $i$ is the unique identifier for each cell $\xi$, and it denotes its order in the generated samples.
  For any resolution level $j$, the resolution level $j+1$ satisfies $\varepsilon_{j+1}=\varepsilon_j/2$. For any two cells $\xi_1,\xi_2$, at resolution levels $r_1,r_2$ respectively, with $r_1<r_2$, and, $\xi_1 \cap \xi_2 \neq \emptyset$, $\xi_1$ will be called as the {\em parent} of $\xi_2$ at resolution $r_1$, and, $\xi_2$ as a {\em child} of $\xi_1$ at resolution $r_2$.
\iflong 
  
    Ideally, one would like to use an {\em ordering} of samples that minimizes the discrepancy to find a counter-example as soon as possible. But discrepancy-optimal orderings take exponential time to compute, and exponential space to be stored (in $n$; see ~\cite{Lin.Yer.Lav:04}). For our work, we use orderings that maximize the {\em mutual distance} between sampled grid points (see ~\cite{Lin.Yer.Lav:04}). The mutual distance of a set $\mathcal{K}$ is defined as
  $\rho_m(\mathcal{K})=\min_{x,y \in \mathcal{K}} \rho(x,y)$. These orderings can be represented by using generator matrices, that are binary matrices of size $n\times n$ and represent {\em bijective linear transformation} over $\mathbb{Z}_2$. Any sample with identifier $i$, at resolution $j$, can be generated by using the ordering recursively on the bit representation of $i$.
  This method can also be used to bias the search towards the unsafe set $\mathcal{T}$, by possibly changing the generator matrix. 
  \else
  For our work, we use orderings that maximize {\em mutual distance} (see~\cite{Lin.Yer.Lav:04}).
  \fi
 In Fig.~\ref{fig:optimal ordering}, we show the samples for few resolution levels in 2 dimensions. As can be seen, cells whose identifiers are close to each other (e.g. Resolution $=2$ or $3$, $i=0,1$) are spaced quite far apart. \iflong This is in sharp contrast to the ordering that one would get by using a naive sampling scheme, like scanning for example.\fi
\iflong
 \begin{figure}[hbt]
\centering
\subfigure[\scriptsize Resolution$=1$]{
\includegraphics[scale=0.2,angle=0]{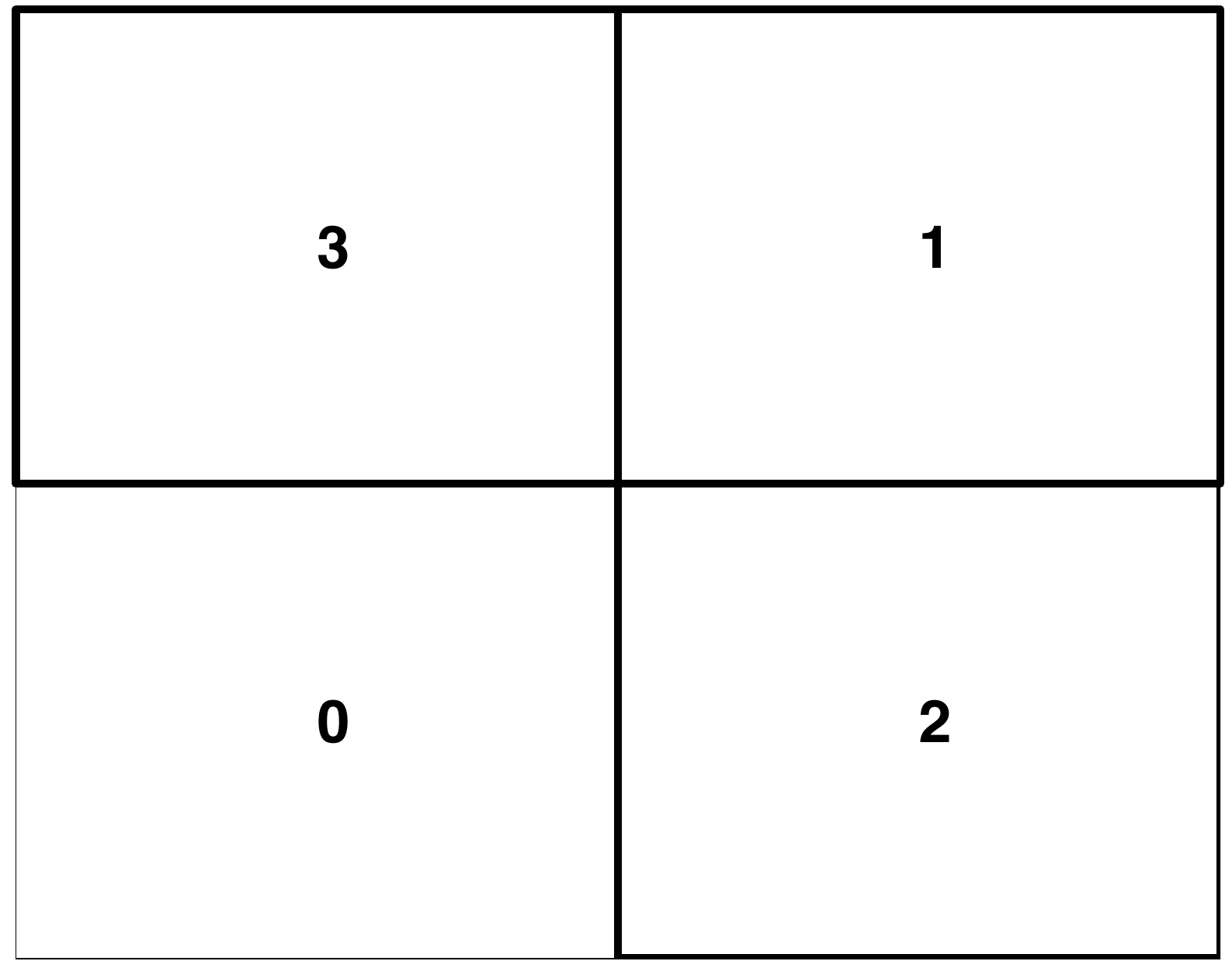}}\hfill
 \subfigure[\scriptsize Resolution$=2$]{
\includegraphics[scale=0.2,angle=0]{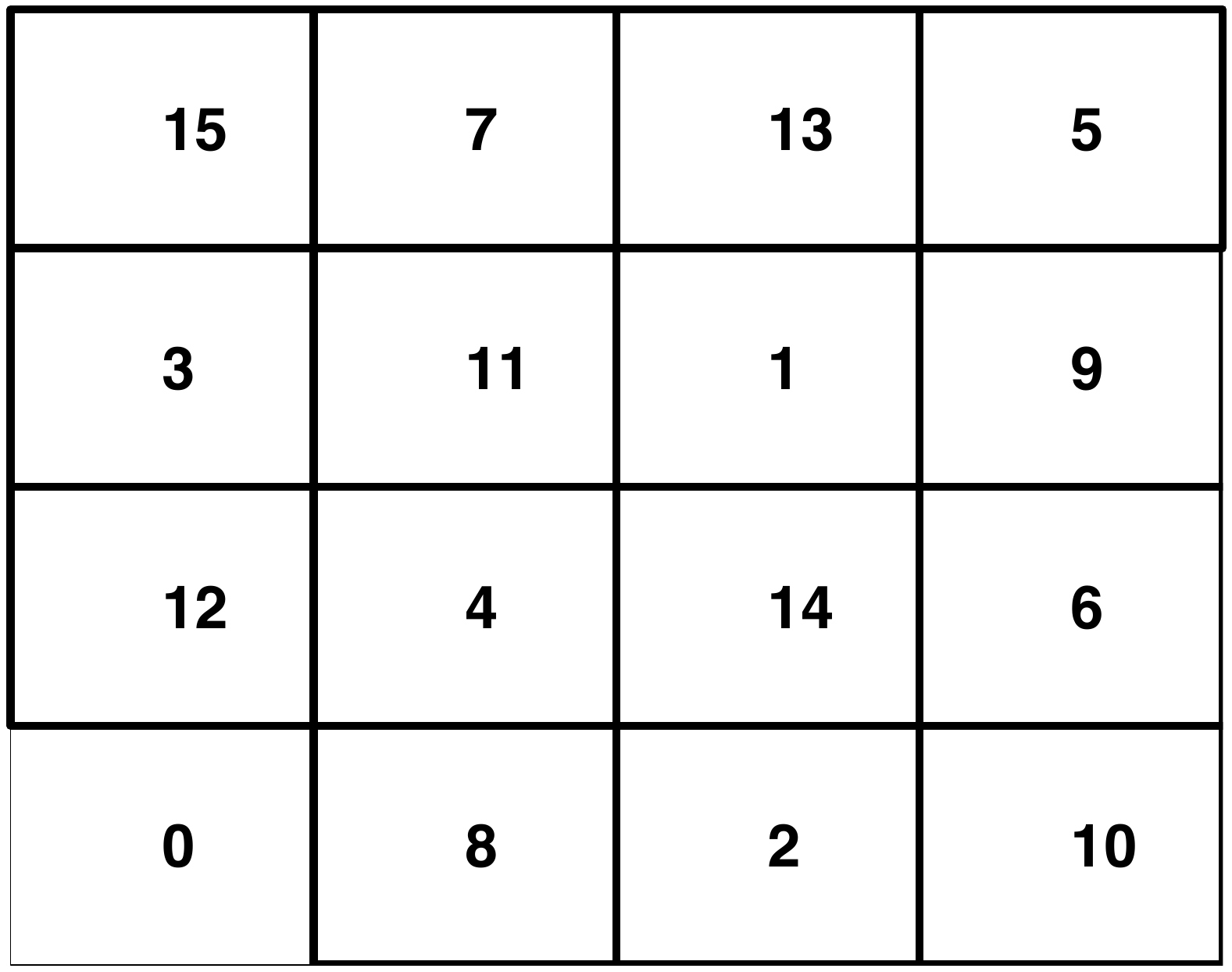}}\hfill
 \subfigure[\scriptsize Resolution$=3$]{
\includegraphics[scale=0.2,angle=0]{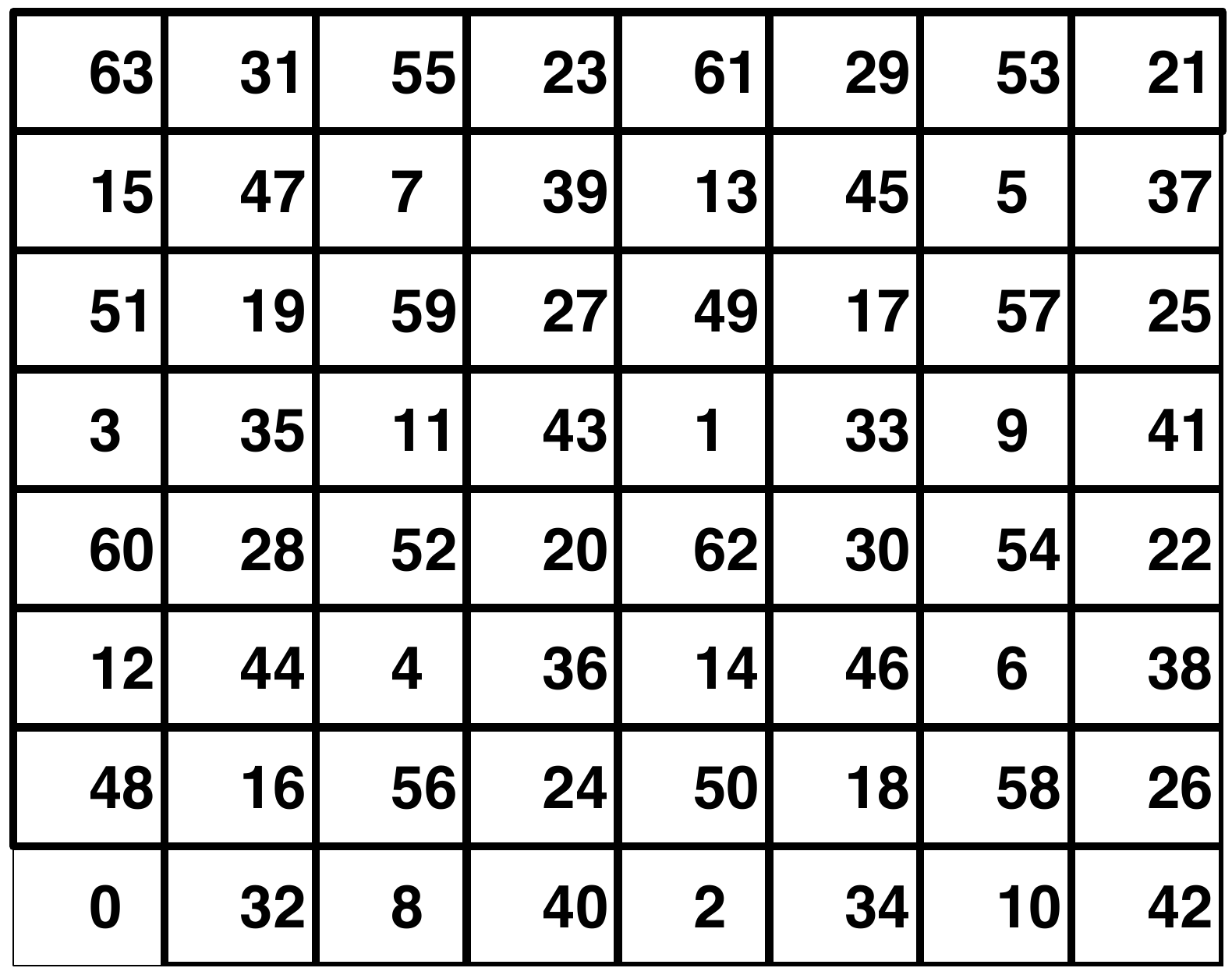}}\hfill
\caption{Ordering of samples based on mutual distance}
\label{fig:optimal ordering}
\end{figure}
\else
\begin{figure}[hbt]
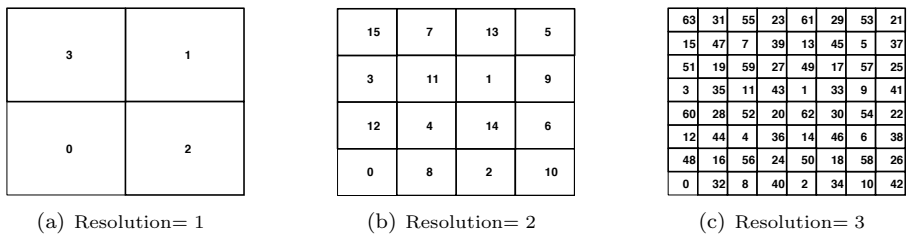

\centering
\subfigure[\scriptsize Resolution$=1$]{
\includegraphics[scale=0.19,angle=0]{figures/ordering_1.pdf}}\hfill
 \subfigure[\scriptsize Resolution$=2$]{
\includegraphics[scale=0.19,angle=0]{figures/ordering_2.pdf}}\hfill
 \subfigure[\scriptsize Resolution$=3$]{
\includegraphics[scale=0.19,angle=0]{figures/ordering_3.pdf}}\hfill
\caption{Ordering of samples based on mutual distance}
\label{fig:optimal ordering}
\end{figure}
\fi
 \section{Algorithms for Resolution-Complete Safety Falsification}
\label{sec:rc algorithms}
 In this section, we first explain the procedure for incremental construction of trajectories used by the algorithms. We then present the algorithms, and in Section~\ref{sec:rc proofs}, prove their resolution completeness.
\subsection{Incremental construction of trajectories}
\label{subsec:explorer dynamics cont}
 As stated in Section~\ref{subsec:control functions cont}, the algorithms
 use piece-wise constant control functions satisfying
 Proposition~\ref{prop:nonemptyint cont}. As discussed in the proofs of
 Propositions~\ref{prop:nonemptyint cont},~\ref{prop:cell size cont}, for a given discrete state $q \in \mathcal{Q}$, the dynamics can be
 simulated by using $x(k)= A_q^k x_0+v$, where $x_0 \in \mathcal{I}(q,\cdot)$, and, the set of feasible
 inputs is given by $v = {\Gamma_q} \tilde u$,
$\tilde u \in \reals^{km}$, $\lVert \tilde u \rVert \leq \tilde l_j(q)= 1- \varepsilon_j(q)
 \lVert {\Gamma_q}^{+} \rVert$. Note, that to be less conservative, we are making the input bound $\tilde l_j(q)$ dependent on $q\in \mathcal{Q}$ and the discretization $\varepsilon_j(q)$. Hence, the algorithms use the input bounds based on the space discretization and dynamics to incrementally build trajectories $k$ steps at a time, by formulating the $k$-step reachability problem as a linear program. To solve these linear programs more efficiently, we use ideas from~\cite{Borelli.Bemporad.ea:03} for {\em multi parametric linear programming}.

 We use the following additional notation in the remaining of this section. $G =
 \cup_{q \in \mathcal{Q}} G(q)$ denotes
 the union of all location specific grids in the
 algorithm. $j_0$ represents the smallest $j$, such that 
$\tilde l_{j_0}(q)= 1- \varepsilon_{j_0}(q) \lVert {\Gamma_q}^{+} \rVert>0, \forall q \in \mathcal{Q}$. For all $j>j_0$, and, $\forall q \in \mathcal{Q}$, $\tilde l_{j}(q)= 1- \varepsilon_{j}(q) \lVert {\Gamma_q}^{+} \rVert$. 
$\varepsilon_j$ is a $Q$-dimensional vector
 with each element denoting value for a given location
 $q$. Since $\tilde l_j(q)$ can be different for each discrete location $q \in \mathcal{Q}$,
 at termination of the algorithms, completeness will be guaranteed with respect to $l_j=\min_{q \in \mathcal{Q}} \tilde l_j(q)$.
\subsection{Breadth First Search with Branch and Bound}
\label{sec:algo bfs-bb}
The proposed algorithm is shown in Fig.~\ref{fig:algo bfs-bb}. We will refer to this algorithm as the \begin{ttfamily} BFS-BB Safety Falsification\end{ttfamily} algorithm. 
A few iterations of the algorithm for a first order discrete-time LTI system with single discrete mode, are shown in Fig.~\ref{fig:bfs-bb algo execution} for a maximum resolution level of $j=3$.
\iflong
\begin{figure}[htb]
\begin{minipage}[t]{0.5\textwidth}
\begin{center}
\includegraphics[scale=0.25,angle=0]{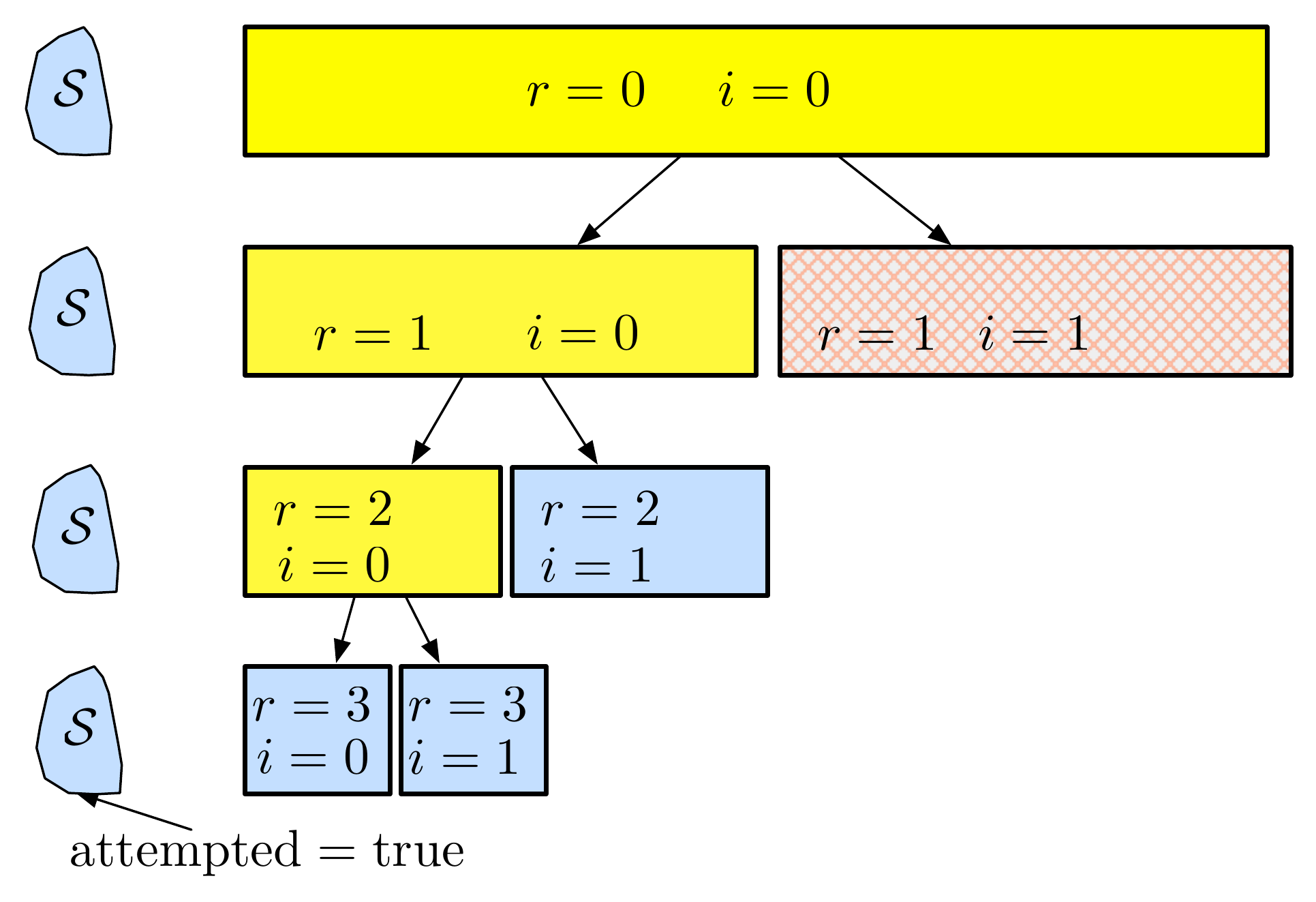}\\
{\scriptsize Exploration starting from $\mathcal{S}$}
\end{center}
\end{minipage}
\begin{minipage}[t]{0.5\textwidth}
\begin{center}
\includegraphics[scale=0.25,angle=0]{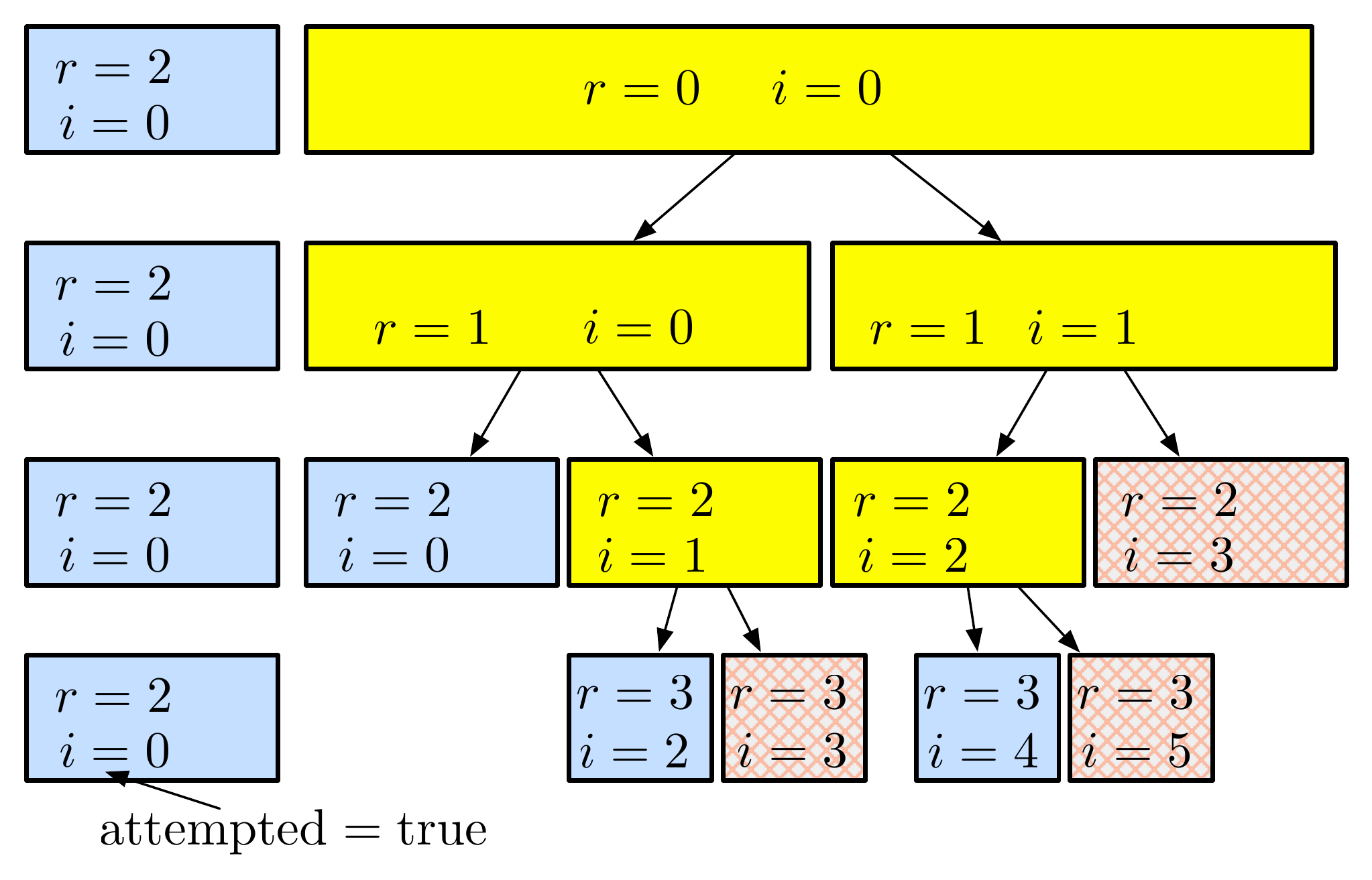}\\
{\scriptsize Exploration starting from $\xi(r=2,i=0)$}
\end{center}
\end{minipage}
\caption{An execution of {\ttfamily BFS-BB Safety Falsification} algorithm}
\label{fig:bfs-bb algo execution}
\end{figure}
\else
\begin{figure}[htb]
\begin{minipage}[t]{0.5\textwidth}
\begin{center}
\includegraphics[scale=0.23,angle=0]{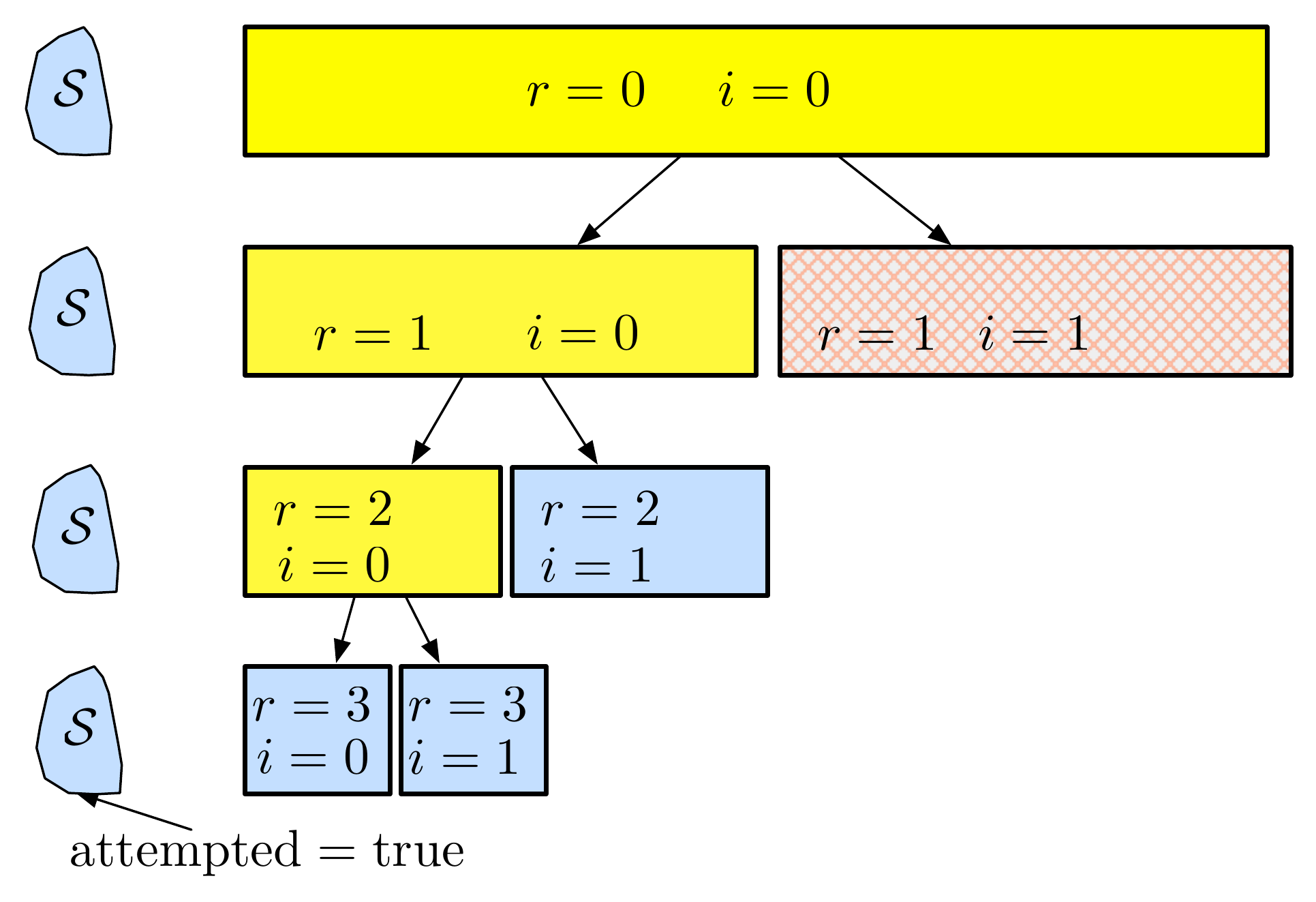}\\
{\scriptsize Exploration starting from $\mathcal{S}$}
\end{center}
\end{minipage}
\begin{minipage}[t]{0.5\textwidth}
\begin{center}
\includegraphics[scale=0.23,angle=0]{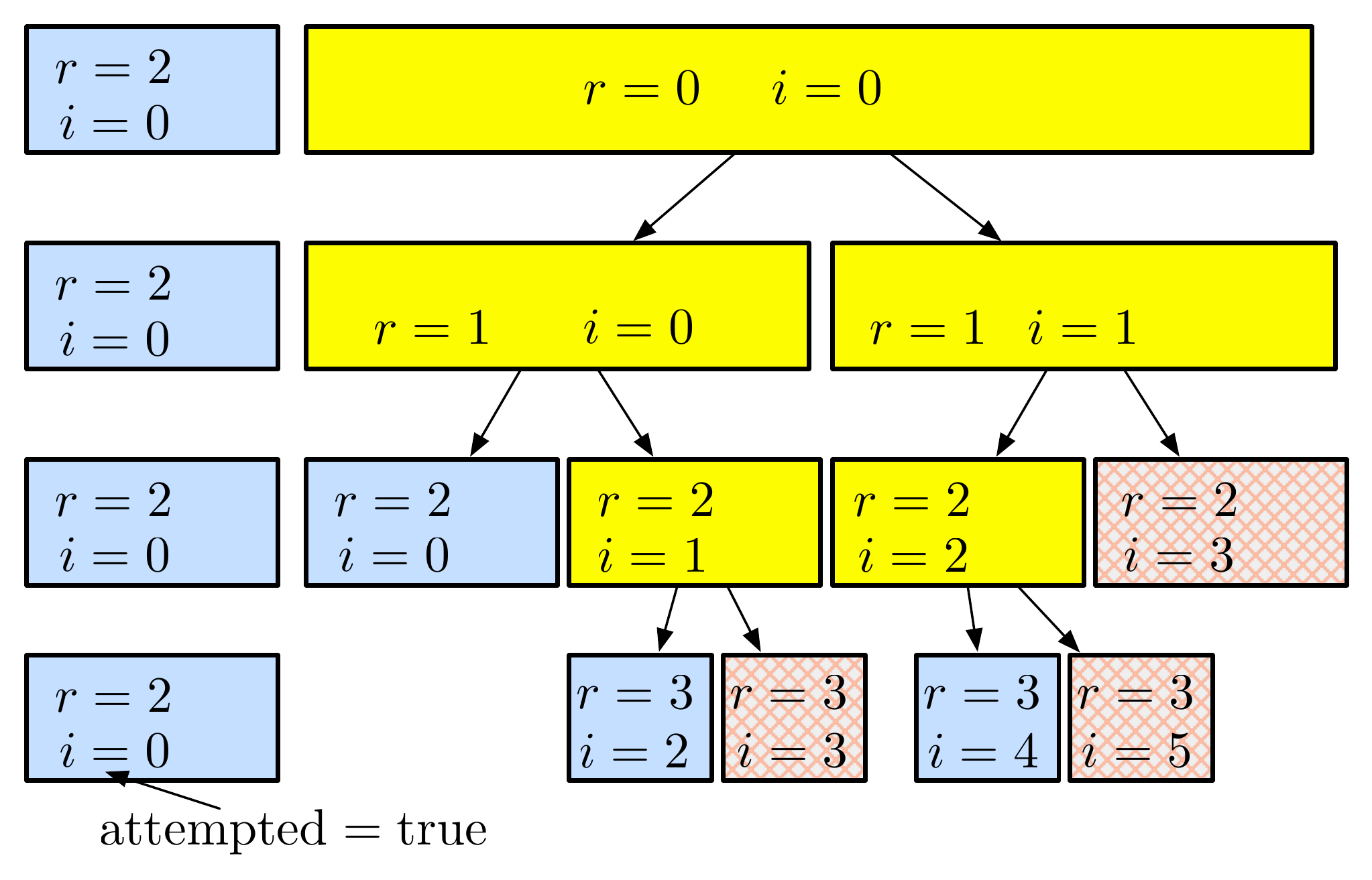}\\
{\scriptsize Exploration starting from $\xi(r=2,i=0)$}
\end{center}
\end{minipage}
\caption{An execution of {\ttfamily BFS-BB Safety Falsification} algorithm}
\label{fig:bfs-bb algo execution}
\end{figure}
\fi
Cells marked in yellow are the ones that need to be explored further based on branch and bound strategy. Cells marked in red are conclusively not reachable from $\xi_f$ (where $\xi_f \in G_f$ is $\mathcal{S}$ at the left and $\xi(r=2,i=0)$ at the right). Cells marked in blue are the ones that are found to be reachable.
\begin{figure}[htb]
\begin{minipage}[t]{0.95\linewidth}
\algsetup{linenodelimiter= }
\scriptsize
\hrule
{\ttfamily BFS-BB Safety Falsification}
($H$, $j_{\mathrm{max}}$, $G $)
\hrule
\begin{algorithmic}[1]
\label{alg:BFS-BB}
\STATE $\{\mathit{locations},\mathit{Found}\} \leftarrow \{\emptyset,\emptyset \}$
\STATE $\{j_0,G_f,\mathit{locations}\} \leftarrow$ {\ttfamily init($H.\mathcal{S}$)}
\COMMENT{Initialization step}
\STATE $j\leftarrow j_0$
\WHILE {($j \leq j_{\mathrm{max}}  \wedge  \neg \mathit{Found}$)}
\FORALL{($q \in \mathit{locations} $)}
\STATE {\ttfamily update\_explorer($q$,$j$)} \COMMENT{Update simulation parameters for current $j,q$}
\FORALL{($\xi_\mathrm{f} \in G_f(q) \wedge \neg \mathit{Found}$)}
\FORALL{($r \in \{1,2,\ldots,j\}$)}
\FORALL{($i\in\{0,\ldots,2^{r n}-1\}$)}
\IF{($\mathit{Found}$)}
\STATE return $(G_f,\mathit{Found},\varepsilon_j)$ \COMMENT{Terminate if unsafe}
\ENDIF
\STATE $\xi_t \leftarrow ${\ttfamily generate\_sample}($i,r,q$) \COMMENT{Sample generation}
\IF{({\ttfamily check\_feasible}($\xi_t$))} 
\STATE $\mathit{success} \leftarrow$ {\ttfamily expand}($\xi_f,\xi_t$) \COMMENT{Expansion step}
\IF{($\mathit{success}$)}
\STATE {\ttfamily add\_node}($\xi_f,\xi_t$) \COMMENT{Update of $G_f(q)$}
\STATE $\mathit{Found} \leftarrow$ {\ttfamily check\_unsafety}($\xi_t$) \COMMENT{Check for unsafety} 
\ENDIF
\ENDIF
\ENDFOR
\ENDFOR
\STATE $\xi_f.\mathit{attempted} \leftarrow \mathrm{true}$ \COMMENT{Mark as attempted}
\ENDFOR
\STATE $\mathit{locations}$.{\ttfamily erase($q$)}
\ENDFOR
\STATE $j+1 \leftarrow G$.{\ttfamily refine(j)} \COMMENT{Refinement step}
\ENDWHILE
\STATE return $(G_f,\mathit{Found},\varepsilon_j)$
\end{algorithmic}
\hrule
\end{minipage}
\hfill 
\caption{{\ttfamily BFS-BB Safety Falsification} algorithm}
\label{fig:algo bfs-bb}
\end{figure}

{\bf Data structure:} Each cell $\xi$ has its identifier $i$, resolution level $r$, and boolean variables $\mathit{attempted, filled}$. $\mathit{attempted}$ is true, if, $\xi \in G_f(q)$ has been attempted for expansion (and the cell is called attempted). The variable $\mathit{filled}$ is true (and the cell is called as filled) if, the cell $\xi$ has size $\epsilon_j(q)$ (for some $j$) and is found to be reachable using $\tilde l_j(q)$, or, all its children at resolution $j+1$ are filled. $G_f$ is implemented as a {\em hash map} to enable quick look up for existing cells in $G_f$.

{\bf Initialization step:} In the first step, the algorithm computes the first feasible resolution level $j_0$ and  initializes $G_f (q)$ $\forall q \in \mathcal{Q}$. Next all $q\in\mathcal{Q}$ are added to $\mathit{locations}$ for which $G_f \neq \emptyset$. This happens in {\ttfamily init($H.\mathcal{S}$)} function in the algorithm.
  
{\bf Exploration step:} The grid is searched for reachable cells with
current bounds on input in a recursive Breadth-First-Search (BFS)
fashion along with Branch and Bound (BB) strategy in the algorithm for each location $q
\in \mathit{locations}$ (after updating the simulation
parameters used by the algorithm for location $q$ in {\ttfamily update\_explorer}($q, j$) function
in the algorithm). First, a filled cell $\xi_f$ is chosen in $G_f(q)$. Next, for all depths $r \in \{1,\ldots,j\}$
samples $\xi_t$ are generated in {\ttfamily generate\_sample}($i,r,q$) function (after checking if $\mathit{Found} = \emptyset$). The function {\ttfamily check\_feasible}($\xi_t$) checks feasibility of $\xi_t$ based on {\em branch and bound} condition from coarse resolution levels, and the fact that its {\em parent} might have been marked as filled already at some coarser resolution level $r' < r$.
 The {\ttfamily expand}($\xi_f,\xi_t$) function attempts to solve the $k$-step reachability problem as discussed in Section~\ref{subsec:explorer dynamics cont}. 
 If it is successful then the {\ttfamily add}($\xi_f,\xi_t$) function adds $\xi_f$ to $G_f(q)$.

{\bf Unsafety check:} Each newly added cell $\xi_t$ in  {\ttfamily add}($\xi_f,\xi_t$) step is checked for intersection with the unsafe set $\mathcal{T}$, if, it is filled, in {\ttfamily check\_unsafety}($\xi_t$) function.

{\bf Discrete transition step:} For a given location $q$, each cell
$\xi \in G(q)$ that is found to
be reachable from $G_f(q)$ is checked for all the outgoing discrete
transitions from $q$. If a guard $\mathcal{G}(q,q')$ is found to be enabled,
then $\mathcal{G}(q,q') \cap \xi$ is added as a filled cell to $G_f(q')$ using the function {\ttfamily add}($\xi, \mathcal{G}(q,q') \cap \xi$) and
$q'$ is added to $\mathit{locations}$. This happens internally in 
{\ttfamily add}($\xi_f,\xi_t$) function.

{\bf Refinement step:} When all the cells $\xi_f \in G_f(q)$ have
been explored for expansion $\forall q \in \mathcal{Q}$, and $\mathit{locations} = \emptyset$, 
then the grid resolution is changed using the relation
$\varepsilon_{j+1}(q) =$$\varepsilon_j(q)/2$,
and the input bounds are
changed from $\tilde l_j(q)$ to $\tilde l_{j+1}(q)$. For all the cells 
$\xi \in G_f(q)$, and, $\forall q \in \mathcal{Q}$,
the variable $\mathit{attempted}$ is reset to false. This happens in
the $G$.{\ttfamily refine(j)} function in the algorithm.

{\bf Termination criteria:} To have a finite termination time, the
refinement procedure is allowed only till $j \leq j_{\mathrm{max}}$, and, no counter example has been found.

\subsection{Resolution-complete Co-RRT}
\label{sec:algo co-rc}
 In this section, we discuss a modified version of the Co-RRT algorithm proposed by us in~\cite{Bhatia.Frazzoli:04} for continuous systems that is resolution-complete. This algorithm can be used for analyzing safety of a continuous system when it starts from equilibrium under the effect of exogenous inputs.  We first state an important lemma regarding reachability of points that are convex combinations of two reachable points.
\begin{proposition}
\label{prop:corc}
 For an origin-stable, continuous system $H$, if $\mathcal{S}=\{0\}$, then for any two reachable points $\psi_1(0,u_1,k_1),\psi_2(0,u_2,k_2), u_1,u_2 \in \mathcal{U}, k_1,k_2 \in \naturals, k_2 \geq k_1$, the convex combination $\psi_{\lambda}=\lambda \psi_1+(1-\lambda) \psi_2, \lambda \in [0,1]$ is also reachable in $k_2$ time steps.
\end{proposition}
\begin{proof}
$\psi_1(0,u_1,k_1)= [B AB \ldots A^{k_1-1}B] u_1$, 
$\psi_2(0,u_1,k_2)= [B AB \ldots A^{k_2-1}B]$$u_2$. $u_1,u_2 \in \reals^{k_1m},\reals^{k_2m}$ respectively. Let $\tilde B_k = [B AB \ldots A^{k-1}B] $, and, $\theta_{c}\in \reals^c$, denote the zero vector. This implies that $\psi_{\lambda} = \lambda \tilde B_{k_2} [u_1' \theta_{m(k_2-k_1)}']' + (1-\lambda) \tilde B_{k_2}u_2$.
 Since $U$ is convex, $u_{\lambda}=\lambda [u_1' \theta_{m(k_2-k_1)}']' + (1-\lambda), u_2 \in \mathcal{U}$.  
\end{proof}
\begin{corollary}
$\mathcal{R}(\mathcal{U})$ is a convex set for a system $H$ (as in Proposition~\ref{prop:corc}).
\end{corollary}
 The proposed algorithm, called as the \begin{ttfamily} Co-RC Safety Falsification\end{ttfamily} algorithm, is shown in Fig.~\ref{fig:algo corc}.
 In Fig.~\ref{fig:corc algo execution}, we show a few iterations of the algorithm for a second order system.
\iflong
\begin{figure}[htb]
\centering
\subfigure[\scriptsize Initialization step]{
\includegraphics[scale=0.25,angle=0]{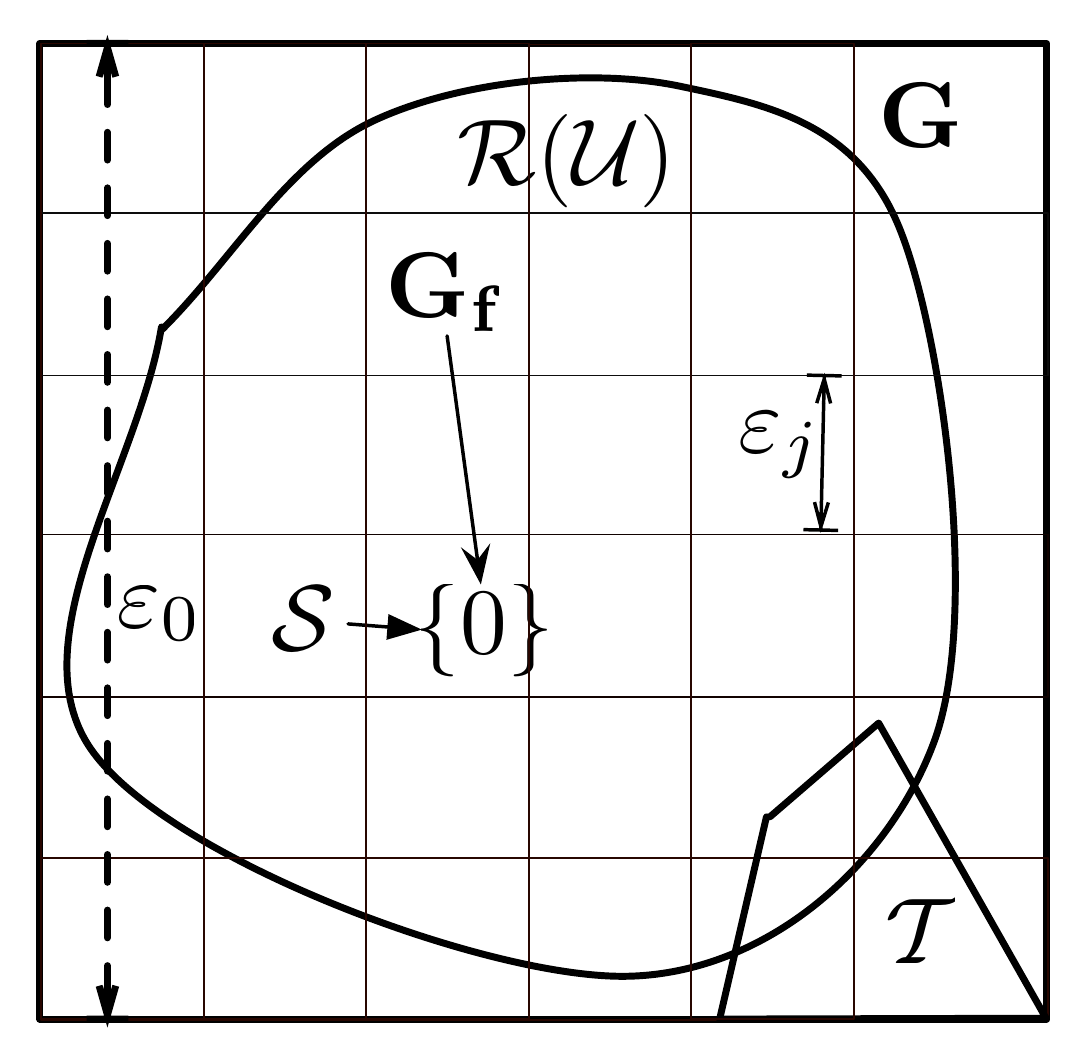}}\hfill
\subfigure[\scriptsize Exploration step]{
\includegraphics[scale=0.25,angle=0]{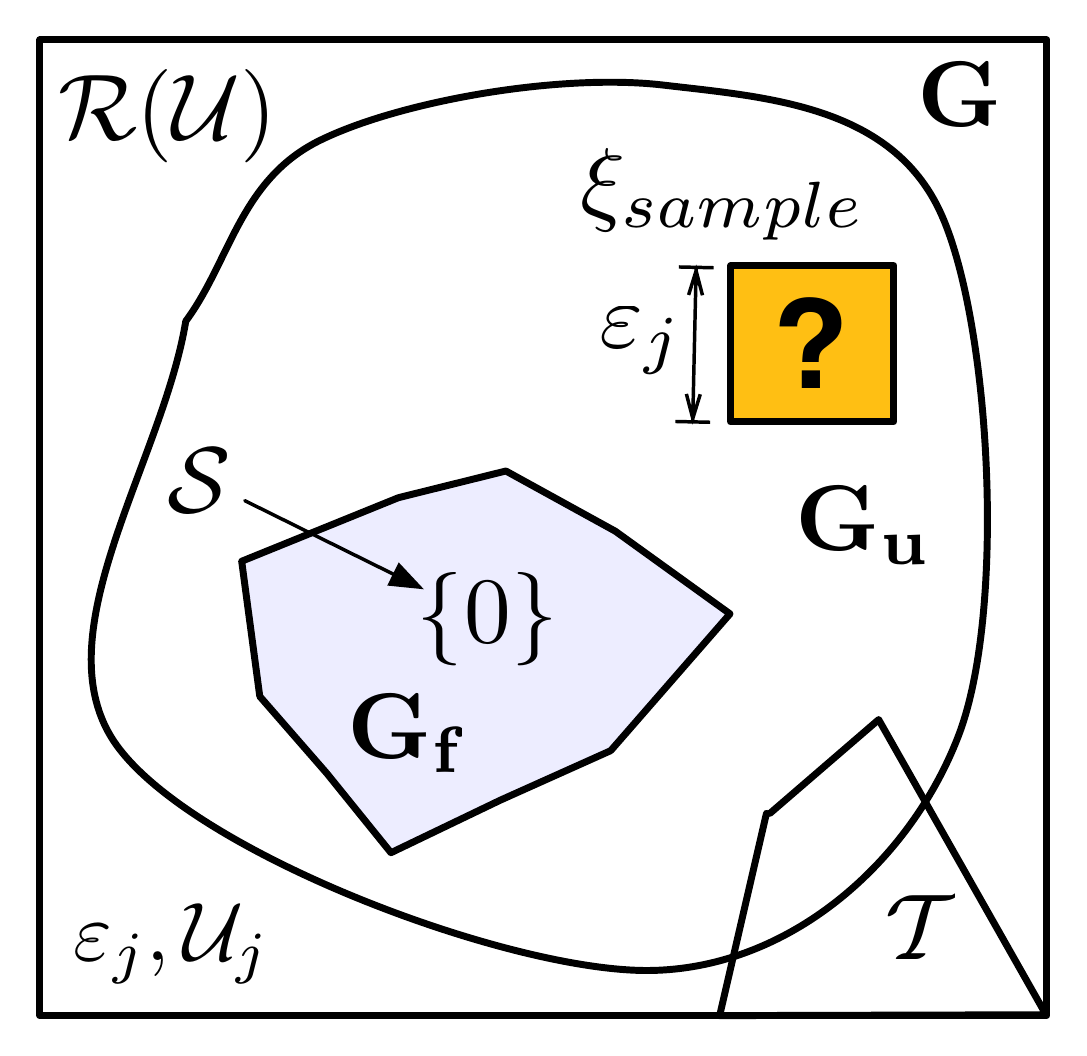}}\hfill
\subfigure[\scriptsize Refinement step]{
\includegraphics[scale=0.25,angle=0]{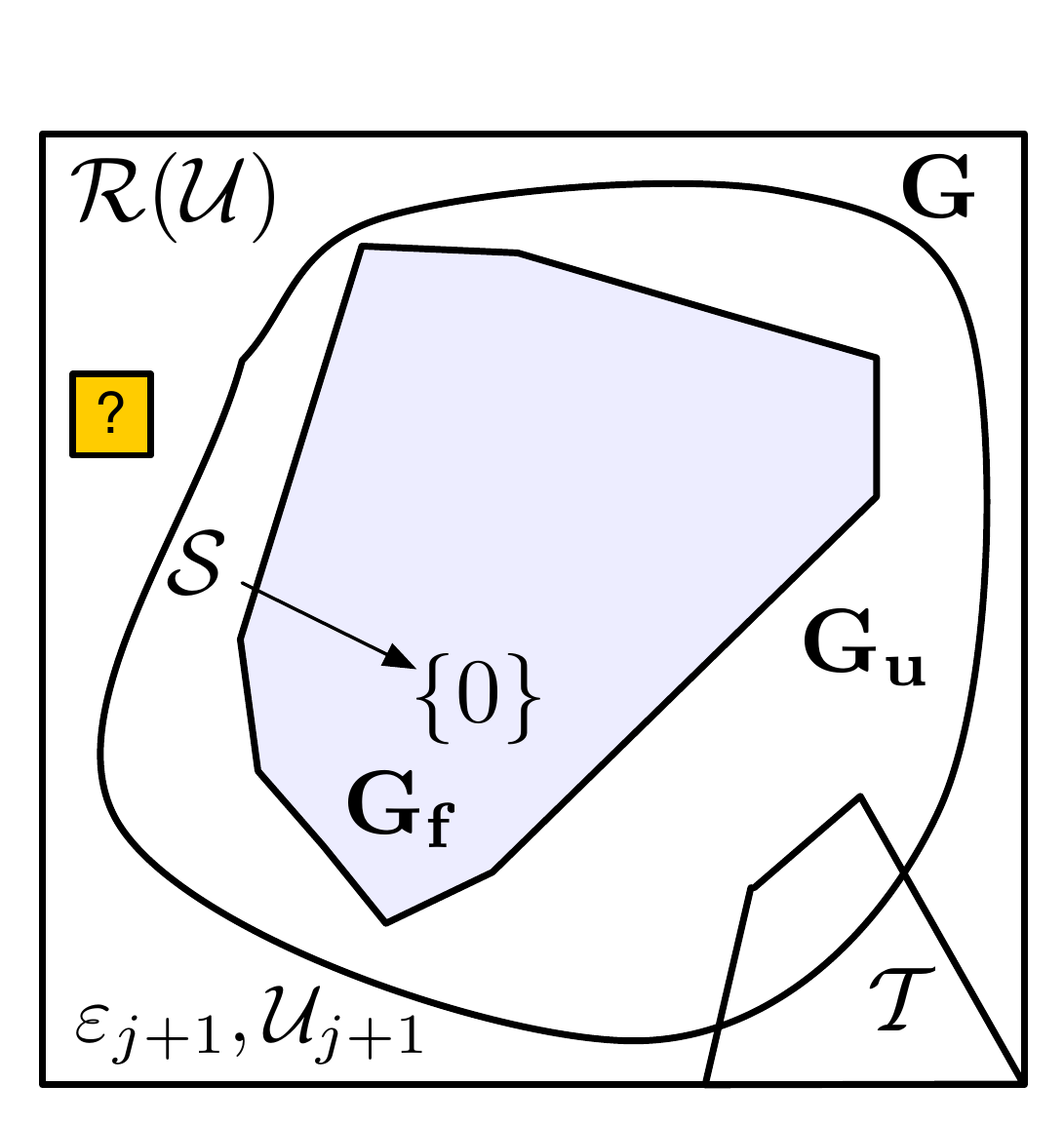}}\hfill
\subfigure[\scriptsize Termination step]{
\includegraphics[scale=0.25,angle=0]{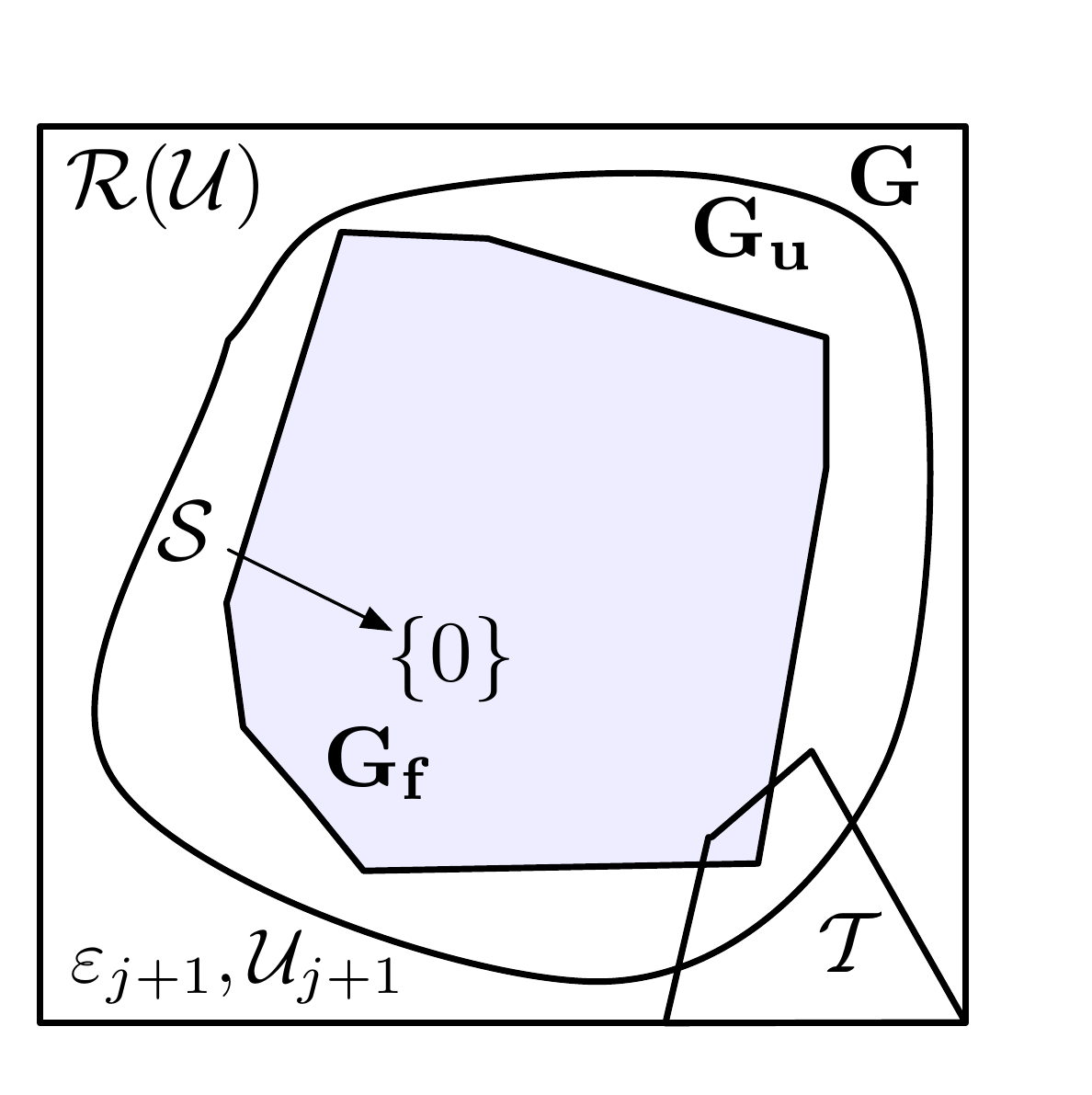}}\hfill
\caption{An execution of {\ttfamily Co-RC Safety Falsification} algorithm}
\label{fig:corc algo execution}
\end{figure}
\else
\begin{figure}[htb]
\centering
\subfigure[\scriptsize Initialization step]{
\includegraphics[scale=0.24,angle=0]{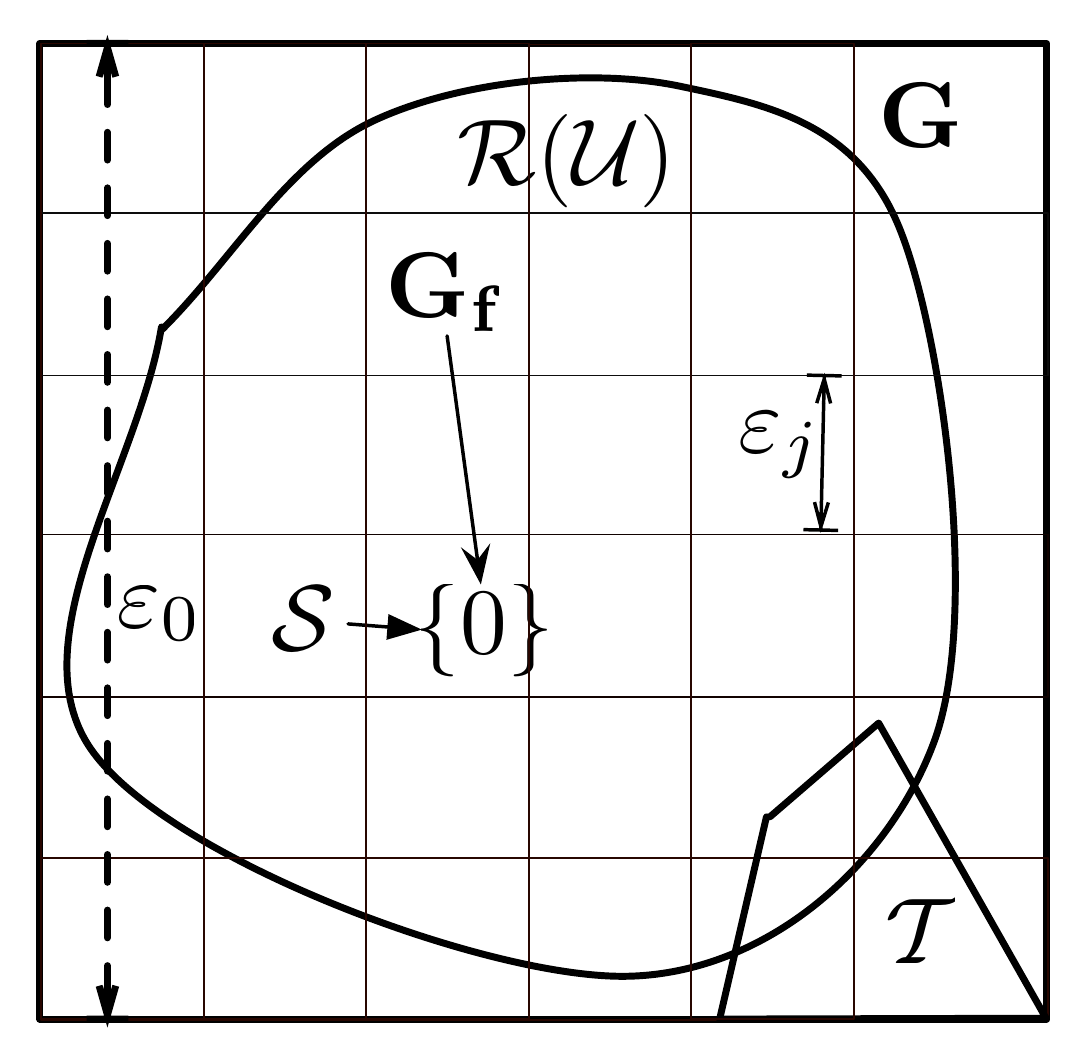}}\hfill
\subfigure[\scriptsize Exploration step]{
\includegraphics[scale=0.24,angle=0]{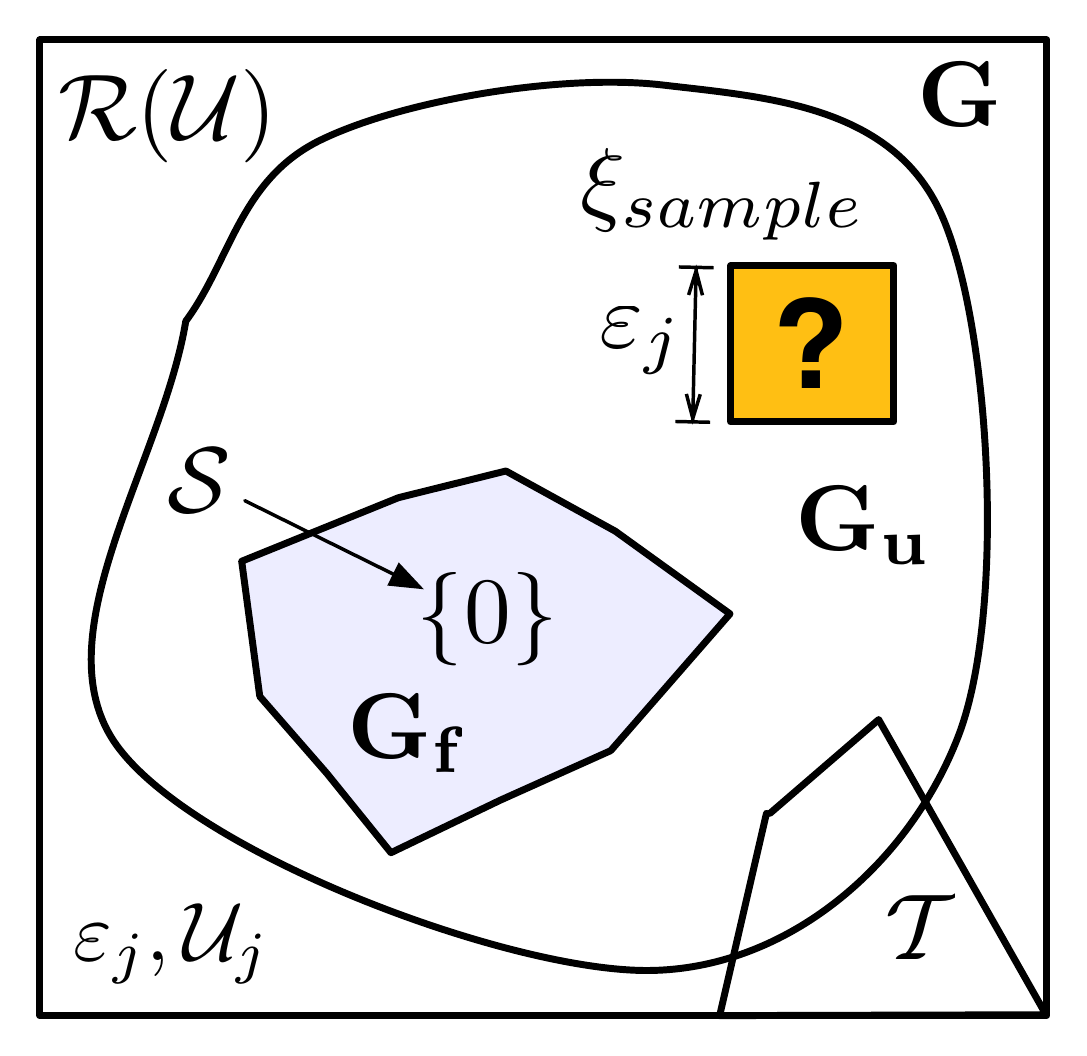}}\hfill
\subfigure[\scriptsize Refinement step]{
\includegraphics[scale=0.24,angle=0]{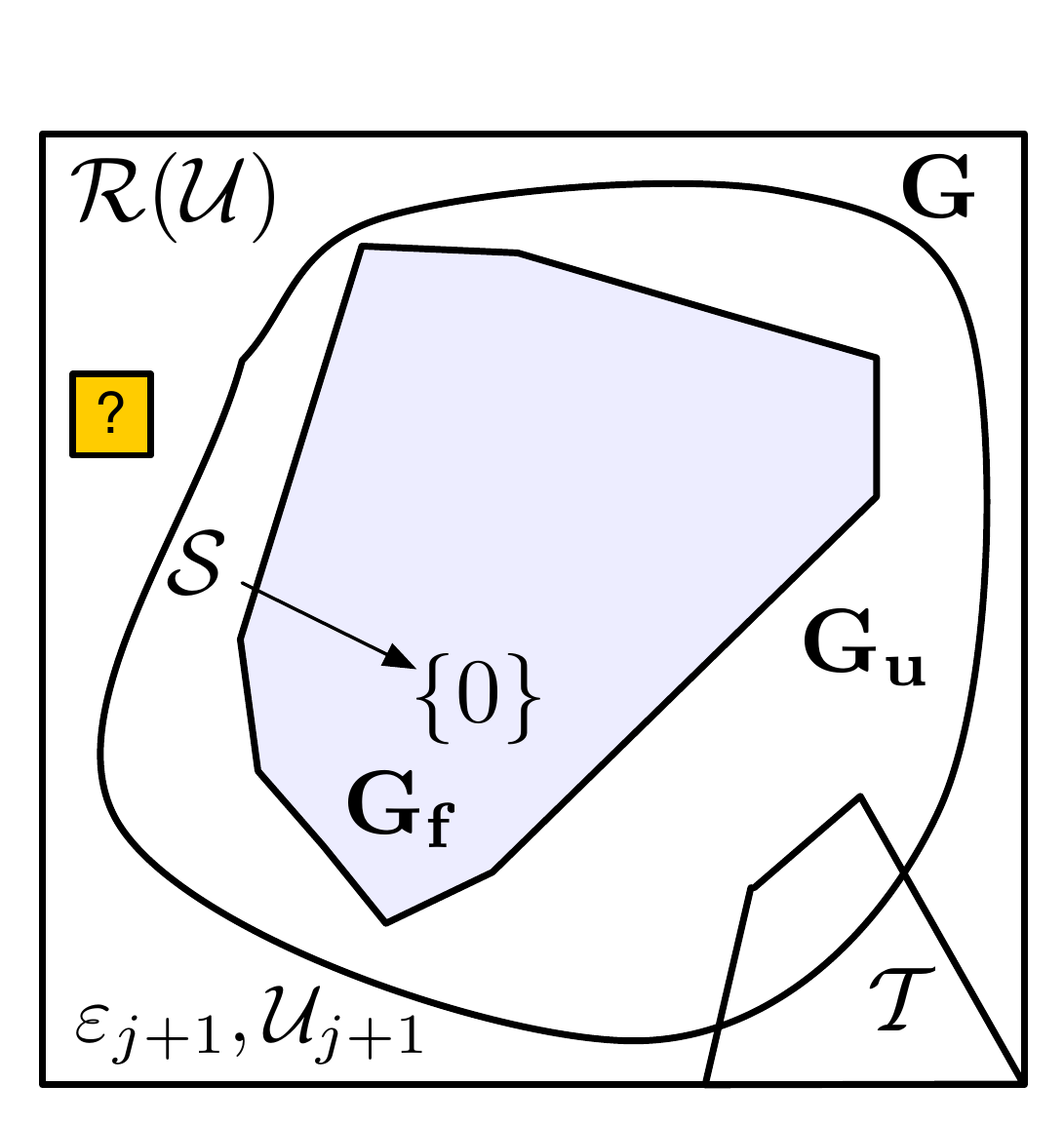}}\hfill
\subfigure[\scriptsize Termination step]{
\includegraphics[scale=0.24,angle=0]{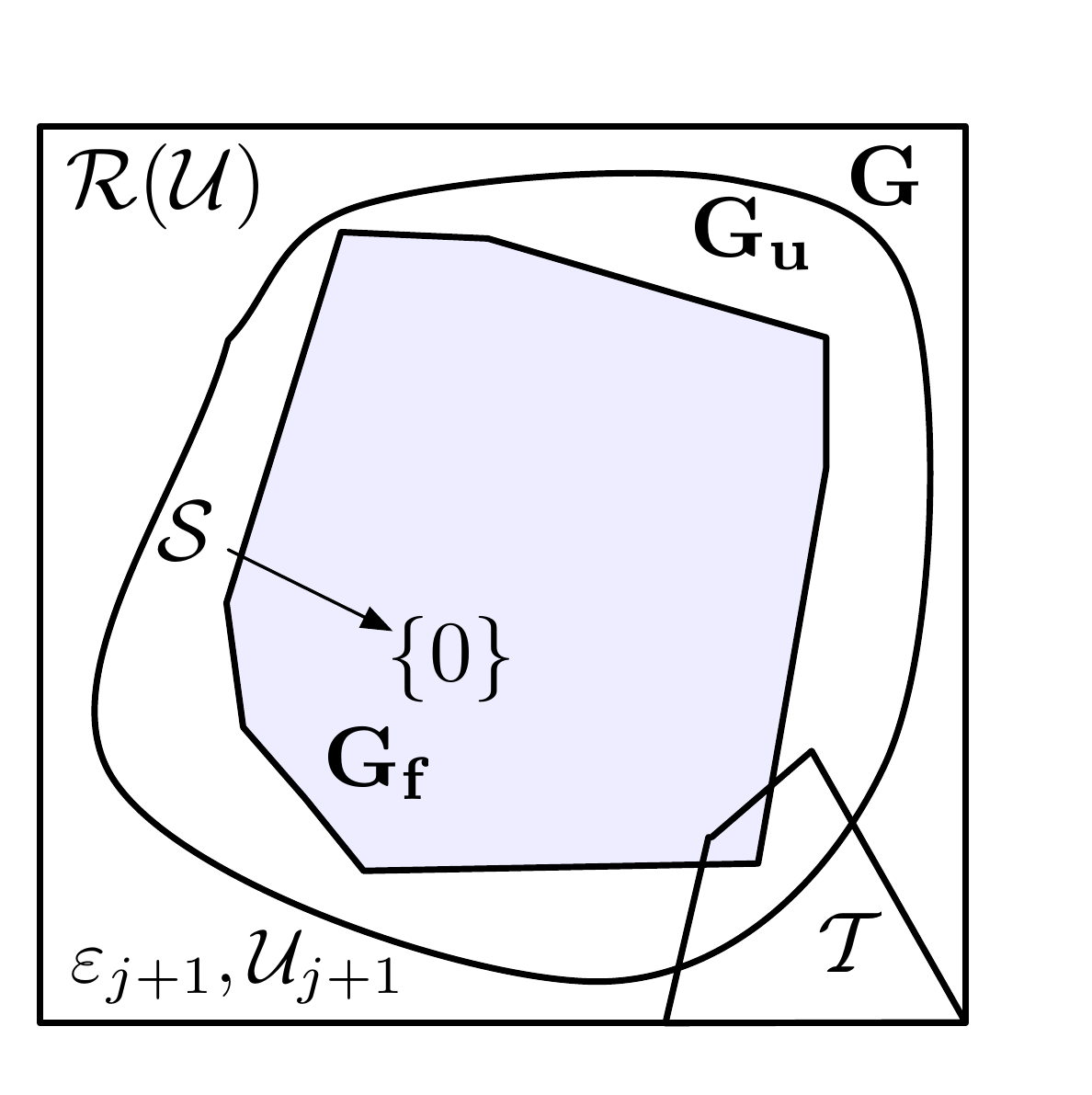}}\hfill
\caption{An execution of {\ttfamily Co-RC Safety Falsification} algorithm}
\label{fig:corc algo execution}
\end{figure}
\fi
\begin{figure}[htb]
\begin{minipage}[t]{0.95\linewidth}
\algsetup{linenodelimiter= }
\scriptsize
\hrule
{\ttfamily Co-RC Safety Falsification}
($H$, $j_{\mathrm{max}}$, $G $)
\hrule
\begin{algorithmic}[1]
\label{alg:corc}
\STATE $\mathit{Found} \leftarrow \emptyset$
\STATE $\{j_0,G_f.\mathrm{Hull}\} \leftarrow G$.{\ttfamily init($H.\mathcal{S}$)}
\COMMENT{Initialization step}
\STATE $j \leftarrow j_0$,$\mathit{change} \leftarrow \mathrm{false}$
\WHILE {($j \leq j_{\mathrm{max}}  \wedge  \neg \mathit{Found}$)}
\STATE {\ttfamily update\_explorer($j$)} \COMMENT{Update the simulation parameters for current j}
\REPEAT
\STATE $\mathit{change} \leftarrow \mathrm{false}$
\FORALL{($r\in\{1,\ldots,j\}$)} 
\FORALL{($i\in\{0,\ldots,2^{rn}-1\}$)} 
\IF{($\mathit{Found}$)} 
\STATE return $(G_f,\mathit{Found},\varepsilon_j)$  \COMMENT{Terminate if unsafe}
\ENDIF
\STATE $\xi_t \leftarrow ${\ttfamily generate\_sample}($i,r$) \COMMENT{Sample generation}
\IF{({\ttfamily check\_feasible}($\xi_t$))}
\STATE $\mathit{success} \leftarrow$ {\ttfamily expand}($G_f.\mathrm{Hull},\xi_t$) \COMMENT{Expansion step}
\IF{($\mathit{success}$)}
\STATE $\mathit{change} \leftarrow \mathrm{true}$ 
\STATE $G_f.\mathrm{Hull} \leftarrow$ {\ttfamily Co}($G_f.\mathrm{Hull},\xi_t$) \COMMENT{Update the Hull}
\STATE $\mathit{Found} \leftarrow$ {\ttfamily check\_unsafety}($G_f.\mathrm{Hull}$) \COMMENT{ Unsafety check}
\ENDIF
\ENDIF
\ENDFOR
\ENDFOR
\UNTIL{($\neg\mathit{change}$)}
\STATE $j+1 \leftarrow G$.{\ttfamily refine(j)} \COMMENT{Refinement step}
\ENDWHILE
\STATE return $(G_f,\mathit{Found},\varepsilon_j)$
\end{algorithmic}
\hrule
\end{minipage}
\hfill 
\caption{{\ttfamily Co-RC Safety Falsification} algorithm}
\label{fig:algo corc}
\end{figure}
\iflong

{\bf Data structure:} The data structure is the same as one in previous algorithm except that $G_f$ now contains a special node $G_f.\mathrm{Hull}$ that contains the information about the convex hull and which is used for actual expansion step, and unsafety checks.

{\bf Initialization step:} In the first step, the algorithm computes the $j_0$, and, $G_f.\mathrm{Hull}$ is initialized to $\mathcal{S}$. This is done in {\ttfamily init($H.\mathcal{S}$)} function in the algorithm.

{\bf Exploration step:} 
The sample generation and search strategy are similar to previous algorithm.
The function {\ttfamily check\_feasible}($\xi_t$) checks feasibility of $\xi_t$ based on the fact that $\xi_t$ may be contained inside $G_f.\mathrm{Hull}$, i.e, $\xi_t \subset G_f.\mathrm{Hull}$, or, its parent may not be reachable with the current $G_f$. In such a case, the function {\ttfamily check\_feasible}($\xi_t$) returns false.
 If it returns true, then the {\ttfamily expand}($G_f.\mathrm{Hull},\xi_t$) function attempts to solve the $k$-step reachability problem as discussed in Section~\ref{subsec:explorer dynamics cont}. 
 If it is successful, then the {\ttfamily Co}($G_f.\mathrm{Hull},\xi_t$) function updates $G_f.\mathrm{Hull}$ in the function  {\ttfamily Co}($G_f.\mathrm{Hull},\xi_f$). The operator {\ttfamily Co} updates $G_f.\mathrm{Hull}$ to the convex combination of $G_f.\mathrm{Hull}$ and vertices of the cell $\xi_t$.

{\bf Unsafety check:} Each time a new cell is added to $G_f$, intersection of $G_f.\mathrm{Hull}$ and $\mathcal{T}$ is checked in the function {\ttfamily check\_unsafety}($G_f.\mathrm{Hull}$).

{\bf Refinement step:} When all the cells at a given resolution level $j$ have been explored (possibly repeatedly) for reachability from $G_f.\mathrm{Hull}$ and no new cell can be reached, the grid resolution is increased according to the relation
$\varepsilon_{j+1} = \max(\varepsilon_j/2,\varepsilon_{min})$,
and the input bounds are
changed from $l_j$ to $l_{j+1}$.This happens in
the $G$.{\ttfamily refine(j)} function in the algorithm.

{\bf Termination criteria:} To have a finite termination time, the
refinement procedure is allowed only till $j\leq j_{\mathrm{max}}$, and, no counter example has been found.
\else

The algorithm works similar to the {{\ttfamily BFS-BB Safety Falsification} algorithm, except that $G_f$ now contains a special node $G_f.\mathrm{Hull}$ that contains the information about the convex hull and which is used for actual expansion step, and unsafety checks. The {\ttfamily Co}($G_f.\mathrm{Hull},\xi_t$) function updates $G_f.\mathrm{Hull}$ in the function  {\ttfamily Co}($G_f.\mathrm{Hull},\xi_f$). The operator {\ttfamily Co} updates $G_f.\mathrm{Hull}$ to the convex combination of $G_f.\mathrm{Hull}$ and vertices of the cell $\xi_t$.
\fi

We next discuss the resolution completeness of the proposed algorithms.
\subsection{Resolution completeness}
\label{sec:rc proofs}
 We first prove two important lemmas required to prove the main result in Theorem~\ref{thm:rc}.
The first lemma proves that for a given maximum resolution
level $j \geq j_0$\footnote{For $j<j_0$ the input bounds are infeasible by choice of $\varepsilon_0$ as discussed in Section~\ref{subsec:explorer dynamics cont}}, the algorithms terminate in finite time, and the
 second lemma proves the required set inclusion for resolution completeness.
\begin{lemma}
\label{lemma:finite time termination cont}
 Consider a LTI discrete-time hybrid system $H$ and a given $j_{\mathrm{max}} \geq j_0$. Then the 
 Safety Falsification algorithms terminate in finite time.
\end{lemma}
\begin{proof}
Consider the \begin{ttfamily}CoRC Safety Falsification\end{ttfamily} algorithm first. The algorithm 
starts with the value $\varepsilon_{j_0}$ initially. If at any stage a counter example is found, finite time
 termination is trivially guaranteed. Otherwise, note that by
 Assumptions~\ref{assmptn: ctb stb cont},~\ref{assmptn:lin eq cont},
G is guaranteed to have a finite volume. Since
there is a given bound $j_{\max}$ on the resolution, it is guaranteed that the
 algorithm will be able to refine the resolution only a finite number
 of times. The sum of the number of cells over all the possible refinements of the grid $G$ is given by
 $N_{\mathrm{cells}}=\frac{2^{n \lceil \varepsilon_{0}/\varepsilon_{j_{\mathrm{max}}}\rceil}-1}{2^n-1}$.  
Hence in the worst case, the algorithm terminates within $N_{\mathrm{cells}}^2$ iterations.
 Now consider the \begin{ttfamily}BFS-BB Safety Falsification\end{ttfamily} algorithm.
 We have $\lvert \mathcal{Q} \rvert$ locations. For each location $q$, let $N_{\mathrm{cells}(q)}$ be the total number of cells over all possible refinements. The number of guards can be no more than $\lvert\mathcal{Q}\rvert^2$. Hence, in the worst case, the algorithm terminates within  $\lvert \mathcal{Q}\rvert ^2 \max_{q\in \mathcal{Q}} N_{\mathrm{cells}(q)}^2$ iterations.
 
\end{proof}
\begin{lemma}
\label{lemma:set inclusion cont}
Consider a LTI discrete-time hybrid system $H$ and a given $j_{\mathrm{max}} \geq j_0$. Then the Safety Falsification algorithms find a
 feasible counter example or else generate
an approximation $\mathcal{R}_{j_\mathrm{max}} $ such that 
$\mathcal{R}^{\circ}(\mathcal{U}_{j_\mathrm{max}})\subseteq \mathcal{R}_{j_\mathrm{max}} \subseteq
\mathcal{R}(\mathcal{U})$.
\end{lemma}
\begin{proof}
 For a given $j_{\mathrm{max}}$, $l_{j_{\mathrm{max}}}$ is fixed. The set inclusion
 $\mathcal{R}_{j_{\mathrm{max}}} \subseteq \mathcal{R}(\mathcal{U})$ follows from the discussion in Section~\ref{sec:basic idea}, and Proposition~\ref{prop:nonemptyint cont},
~\ref{prop:cell size cont}. This
 guarantees feasibility of the counter examples found by the algorithms. For a location $q \in \mathcal{Q}$,
if a point is found reachable by the algorithms in a cell
$\xi$, then the algorithms mark that cell as reachable based on the
relaxation of input bounds from $l_{j}(q)$ to $1$, where $j_0\leq j \leq j_{\mathrm{max}}$. As a
result we are guaranteed that when the algorithms terminate, the set inclusion
$\mathcal{R}^{\circ}(\mathcal{U}_{j_{\mathrm{max}}})\subseteq 
\mathcal{R}_{j_{\mathrm{max}}} $ also holds true. Taking convex combinations of reachable cells in the  \begin{ttfamily}CoRC Safety Falsification\end{ttfamily} algorithm does not violate this inclusion. 
\end{proof}
\begin{theorem}
\label{thm:rc}
Consider a LTI discrete-time hybrid system $H$ as in
Lemmas~\ref{lemma:finite time termination cont},~\ref{lemma:set
  inclusion cont}. Then the Safety Falsification algorithms are resolution complete for the system $H$.
\end{theorem}
\begin{proof}
 Let $j_{\mathrm{max}} \geq j_0$ be given. Proposition~\ref{prop:control approx cont} guarantees the existence of required sequence of family of control functions.
 From Lemma~\ref{lemma:set inclusion cont} we know that if a counter example is found it is feasible. If no counter example is found,
 then it implies that there doesn't exist one (using the class of
 control functions $\mathcal{U}_{j_{\mathrm{max}}}$), from the set inclusion
 proved in Lemma~\ref{lemma:set inclusion cont}.
Moreover,, 
Lemma~\ref{lemma:finite time termination cont} guarantees that the
algorithms will terminate in a finite time. 
\end{proof}
\section{Simulation Results}
 \label{sec:experiments}
In this section, we present the simulations results obtained on different discrete time systems.
 The implementation has been carried out in C++
on a Pentium 4, 2.4 GHz machine, with 512 MB of RAM.
We will examine the performance of algorithms presented in Section~\ref{sec:algo bfs-bb} (called as BFS-BB), Section~\ref{sec:algo co-rc} (called as CoRC) and the one presented in~\cite{Bhatia.Frazzoli:07}, which is based on Depth-First-Search with Branch and Bound strategy (called as DFS-BB).
\subsection{Safety falsification of a discrete-time fifth-order system}
 The example presented in this section is mainly intended to
 investigate performance of different algorithms for problems in moderate dimensions. We consider a fifth-order system, with dynamics specified as: 
$x(i+1) = Ax(i)+Bu(i)$, $ A = 0.6065$ I, $B=3.935$ I. The initial set is 
$\mathcal{S}=\{0\}$. The simulation
 parameters are as follows: $\varepsilon_0 = 27.2, j_{\mathrm{max}} = 4$. This corresponds to $\lVert u \rVert \leq 0.5679 $. 
 The unsafe set is a hypercube of size 0.90 and is given by 
$\mathcal{T}= x + [-0.45,0.45]^5$.
 We did 50 test runs each with an unsafe set randomly centered at $x$, where $x \in
 [-7,7]^5$ was chosen using a uniform distribution.
  
  The performance results using different algorithms are shown in Fig.~\ref{fig:perf 5d}. The performance results indicate that when ever a counter example was found, all the algorithms terminated within 5 minutes. The BFS-BB algorithm falsifies safety sooner than other two algorithms. From the results, it is difficult to comment if CoRC performs better than DFS-BB or not. We have also done some profiling of the CoRC algorithm, and have found that almost 70\% of the time is spent in removing the redundant vertices of the hull in the {\ttfamily Co}($G_f.\mathrm{Hull},\xi_f$) function. In our current implementation, we reconstruct the convex hull from scratch every time the {\ttfamily Co}($G_f.\mathrm{Hull},\xi_f$) function is called. We are working on using software libraries that support incremental construction of convex hulls.
\begin{figure}
\centering
\subfigure[\scriptsize Performance of DFS-BB]{
\includegraphics[scale=0.23,angle=0]{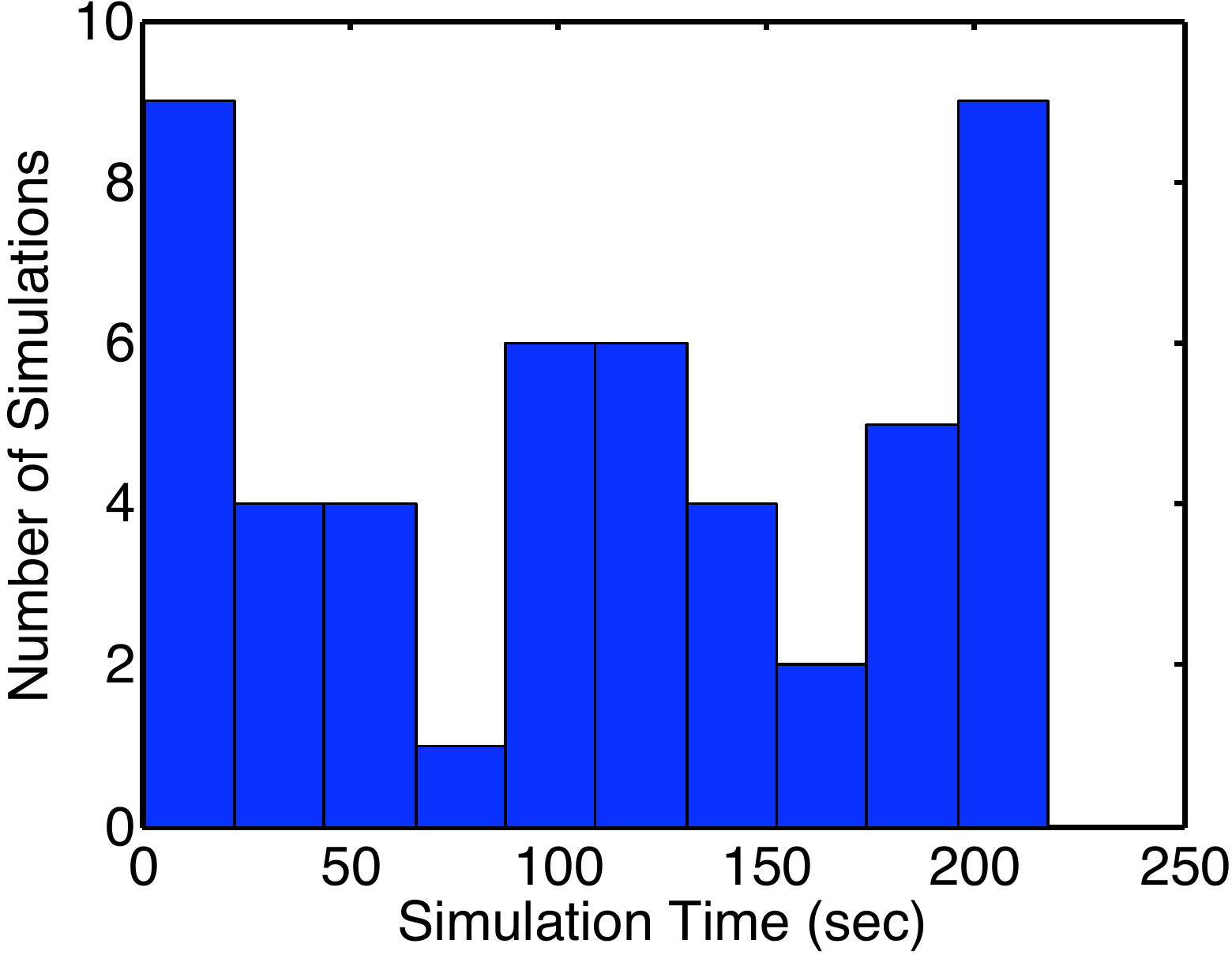}}\hfill
\subfigure[\scriptsize Performance of BFS-BB]{
\includegraphics[scale=0.23,angle=0]{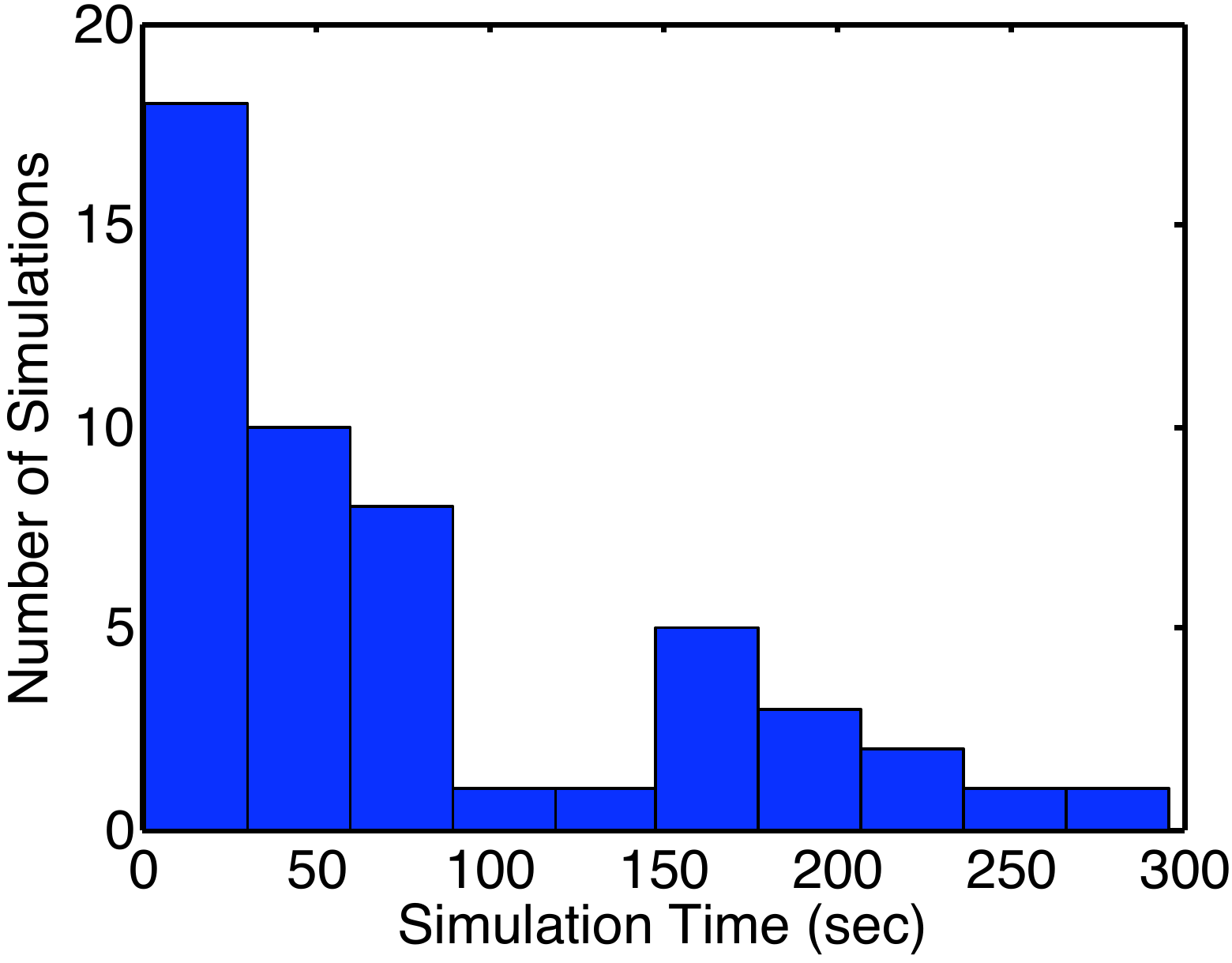}}\hfill
\subfigure[\scriptsize Performance of CoRC]{
\includegraphics[scale=0.23,angle=0]{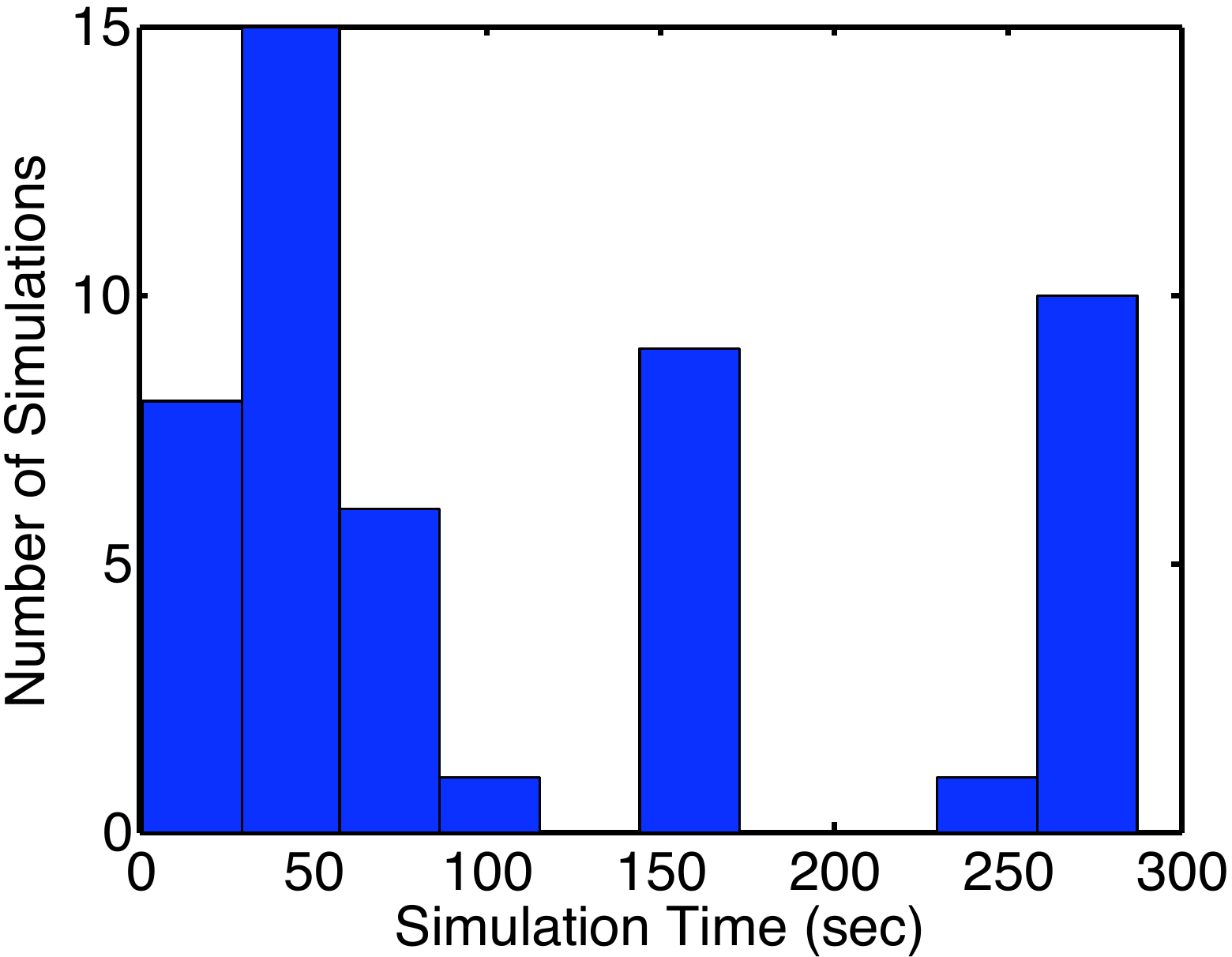}}
\caption{Performance comparison of different algorithms for fifth-order system} 
\label{fig:perf 5d}
\end{figure}
\subsection{Safety falsification of a discrete-time second-order hybrid system}
 We now consider a second-order hybrid system with two discrete states. 
The example is interesting because of the fact that the
 guards and the invariants are all the same, and hence there is a
 potential problem of cycles. The dynamics are specified as 
$x(i+i) = A_qx(i)+B_qu(i), q \in \{1,2\}$, with,
$A_1 = [0.679\; 0.404;-0.674\; 0.140]$,
$B_1 = [0.440; -0.213]$,
$A_2=[0.679\;-0.404;0.674\;0.140]$,
$B_2 = [0.3486;-0.1628]$.
$\mathcal{I}(1,\cdot) =
 \mathcal{I}(2,\cdot)=\mathcal{G}(1,2)=\mathcal{G}(2,1) = [-2,2]^2$. 
The initial set is $\mathcal{S}=0 \times [-0.1248,0.1248]^2$. The simulation
parameters are $\varepsilon_0=4, j_{\mathrm{max}} =7$. 
Corresponding to this $\lVert u \rVert \leq 0.792$ for $ q=0$, and 
$\lVert u \rVert \leq 0.869$ for $ q=1$. Hence, for the case when no counter
example is found, completeness will be guaranteed for $\lVert u \rVert \leq 0.792$.

The simulations are run for two cases based on the size
 of the unsafe set. In Case I, the unsafe set is small which
 makes it more difficult to be found by the explorer. However, quite
 often the unsafe set is specified as a
 large set (e.g. a half space representing the separation between two cars or
 aircrafts). To evaluate the performance of the algorithm for such
 cases, we present test runs in Case II. We did 200 runs for both the cases.\\
{\bf Case I: Small unsafe sets:}
For this case, the unsafe set is given by:
$\mathcal{T}= q_r \times \{x + [-0.1,0.1]^2\}$, where $x \in
 [-1.5,1.5]^2$, was chosen using a uniform distribution and $q_r \in \{0,1\}$.\\
{\bf Case II: Large unsafe sets:}
For this case, we chose the unsafe set as a randomly oriented half
space, given by: $\mathcal{T}=q_r\times\{(x_1,x_2): mx_1+ax_2\leq c\}$, where, 
$m\in[-3.73,3.73], c\in[-1.9,0]$,$q_r \in \{0,1\}$, and $a = -\operatorname{sgn}(c)$. $c,m$ are chosen from a uniform distribution such that $\mathcal{S}\cap \mathcal{T} = \emptyset$.
  
 Fig.~\ref{fig:unsafety_example_hcube-unsafety_perf_hcube} and
 Fig.~\ref{fig:unsafety_example_hspace-unsafety_perf_hspace} show
 performance results for the two cases along with a run when a counter
 example was found. Both the DFS-BB and BFS-BB algorithms falsify safety sooner for Case II.
 The BFS-BB falsifies safety in less than 10 seconds for 68 cases, compared to 54 for DFS-BB in Case I, and 164 times compared to 149 times for Case II. 
\begin{figure}
\centering
\subfigure[\scriptsize Example for Case 1]{
\includegraphics[scale=0.23,angle=0]{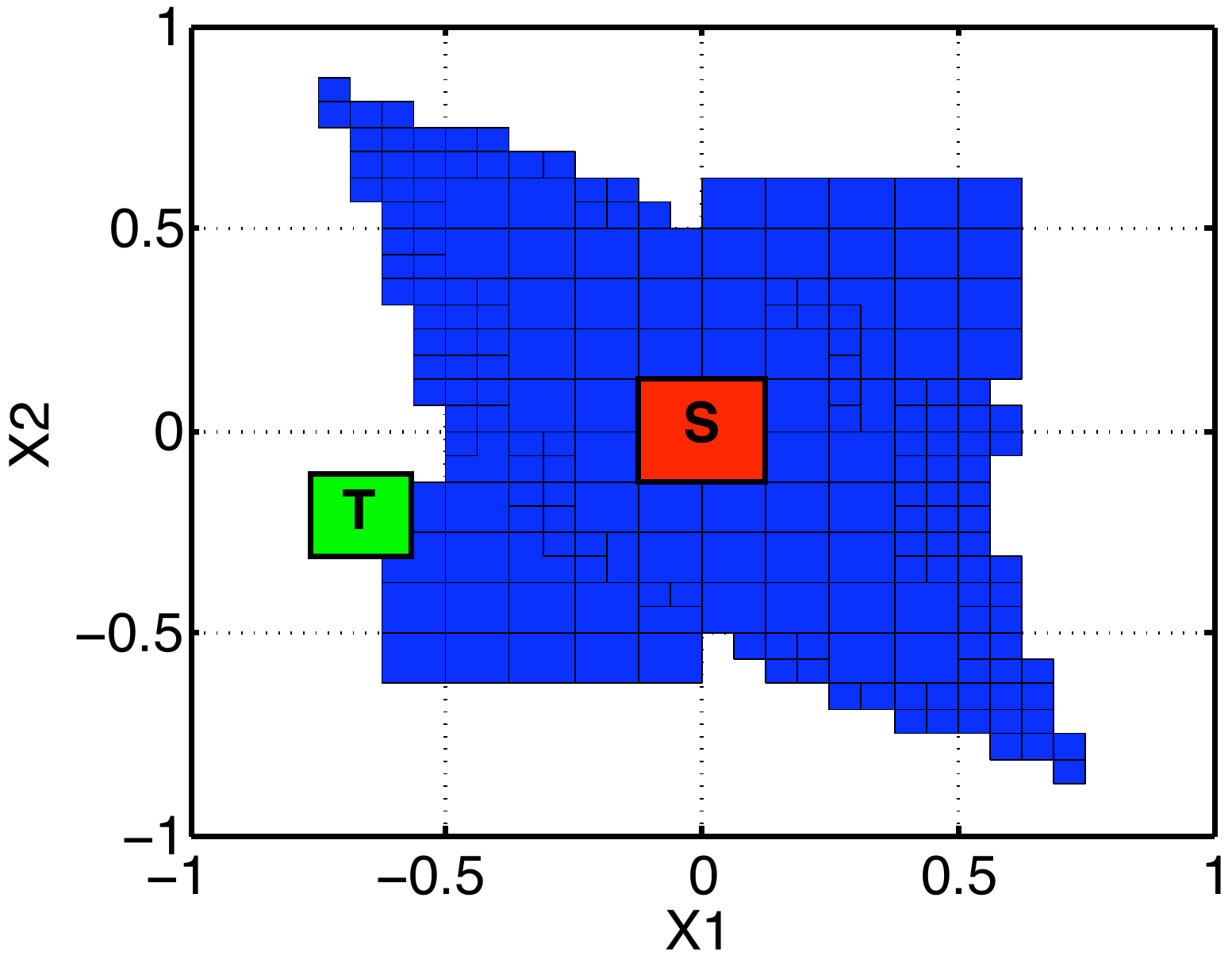}}\hfill
\subfigure[\scriptsize Performance of DFS-BB ]{
\includegraphics[scale=0.23,angle=0]{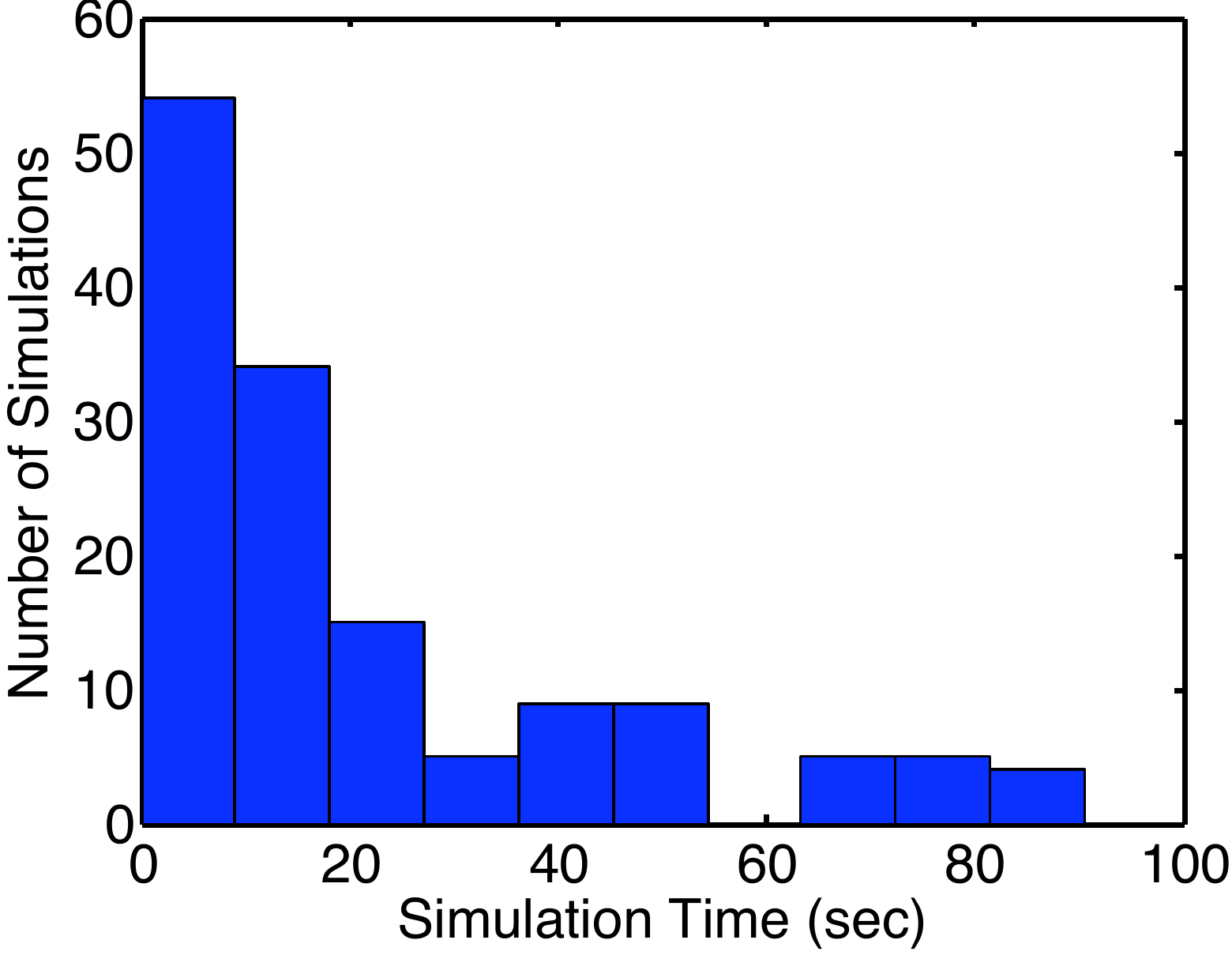}}\hfill
\subfigure[\scriptsize Performance of BFS-BB]{
\includegraphics[scale=0.23,angle=0]{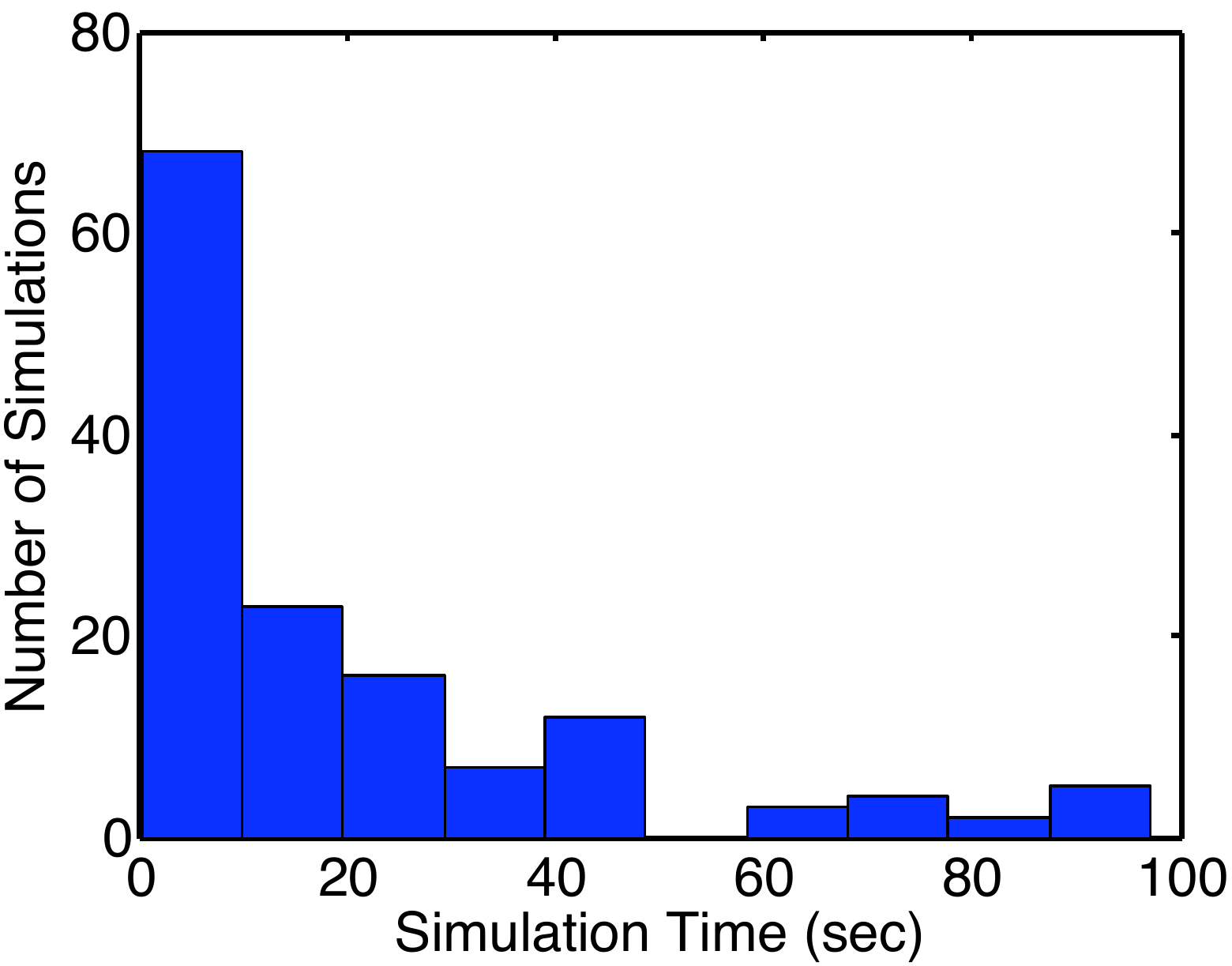}}
\caption{Safety falsification example and performance for Case I}
\label{fig:unsafety_example_hcube-unsafety_perf_hcube}
\end{figure}
\iflong \else \vspace{-0.5 in} \fi
\begin{figure}
\centering
\subfigure[\scriptsize Example for Case II]{
\includegraphics[scale=0.23,angle=0]{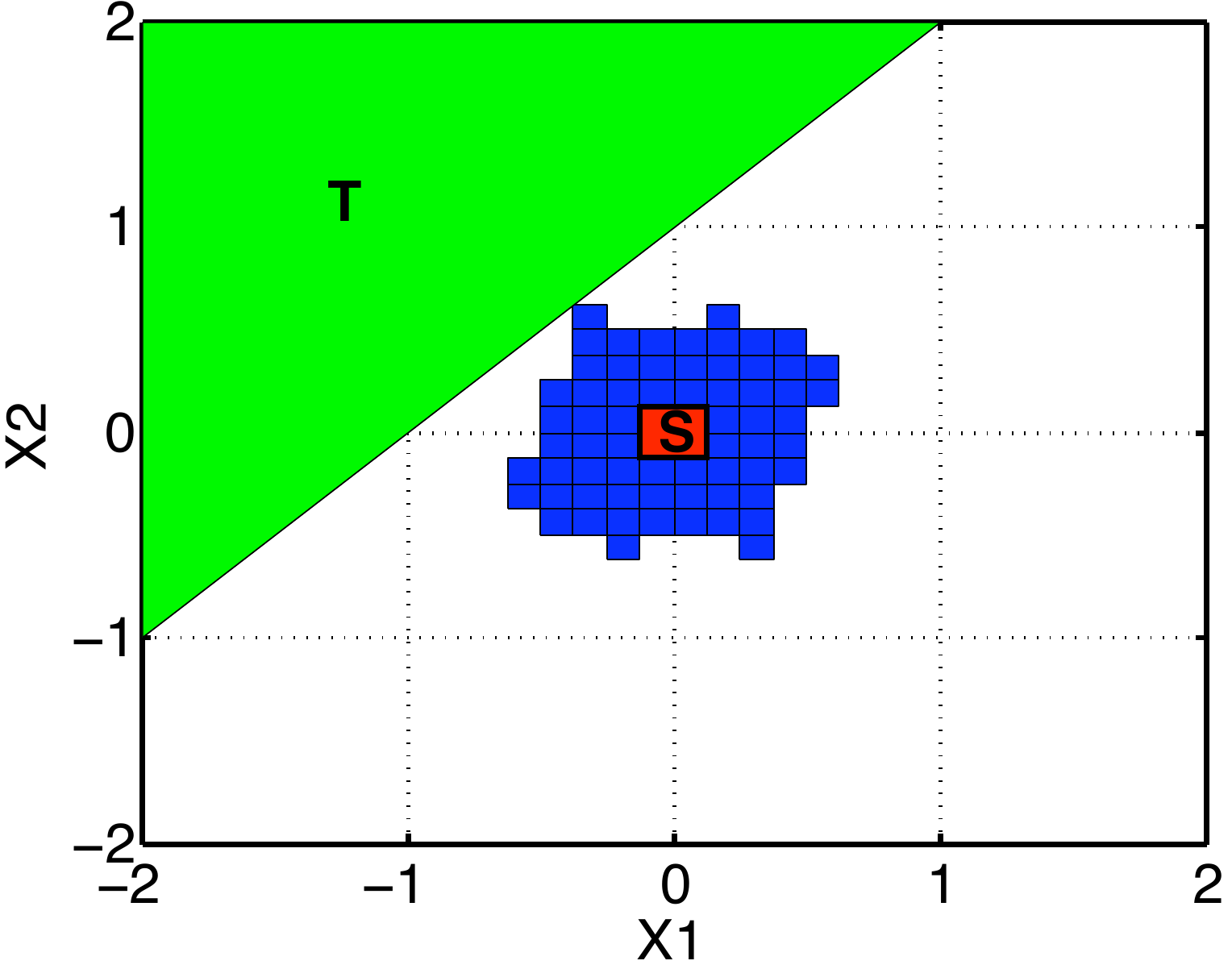}}\hfill
\subfigure[\scriptsize Performance of DFS-BB]{
\includegraphics[scale=0.23,angle=0]{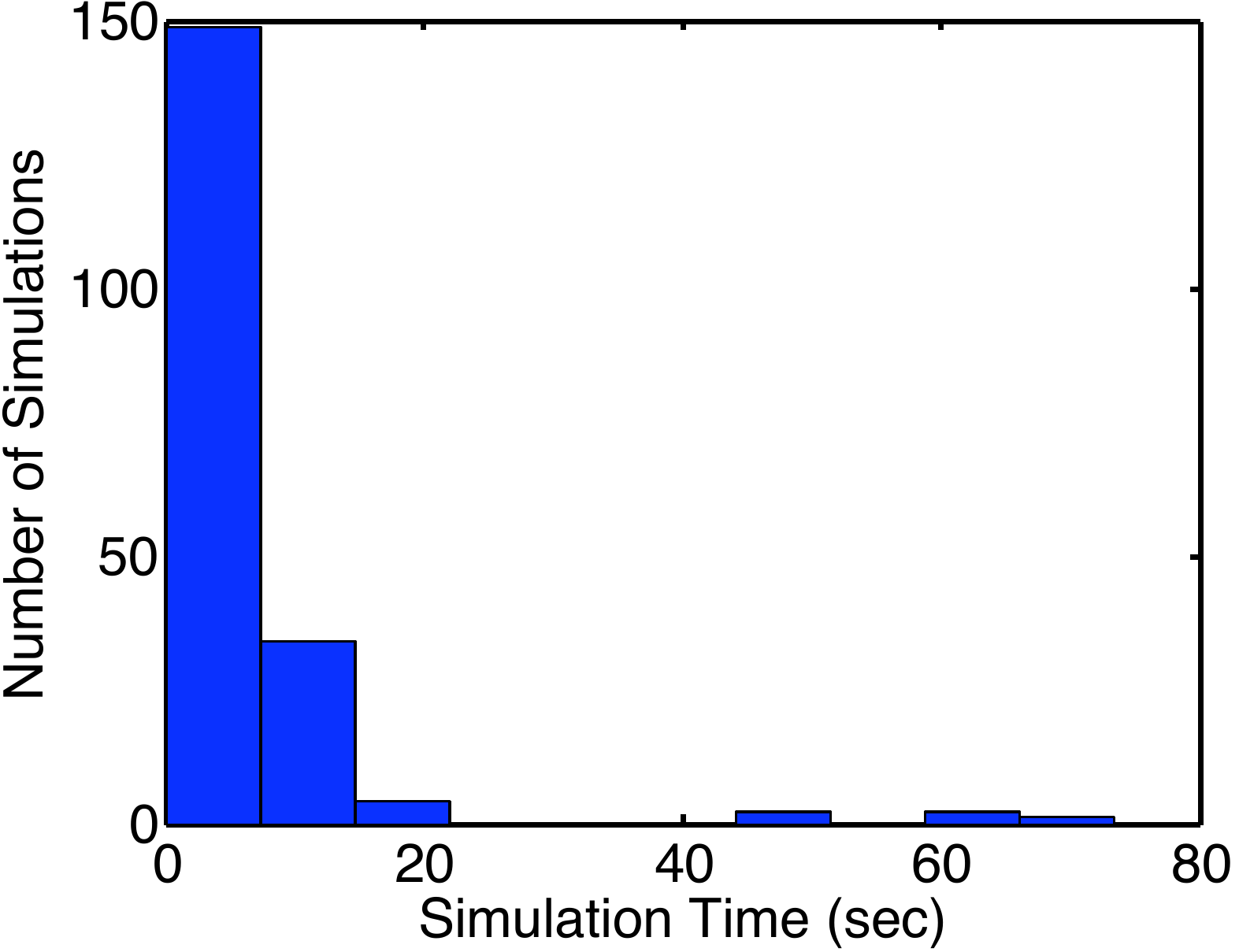}}\hfill
\subfigure[\scriptsize Performance of BFS-BB]{
\includegraphics[scale=0.23,angle=0]{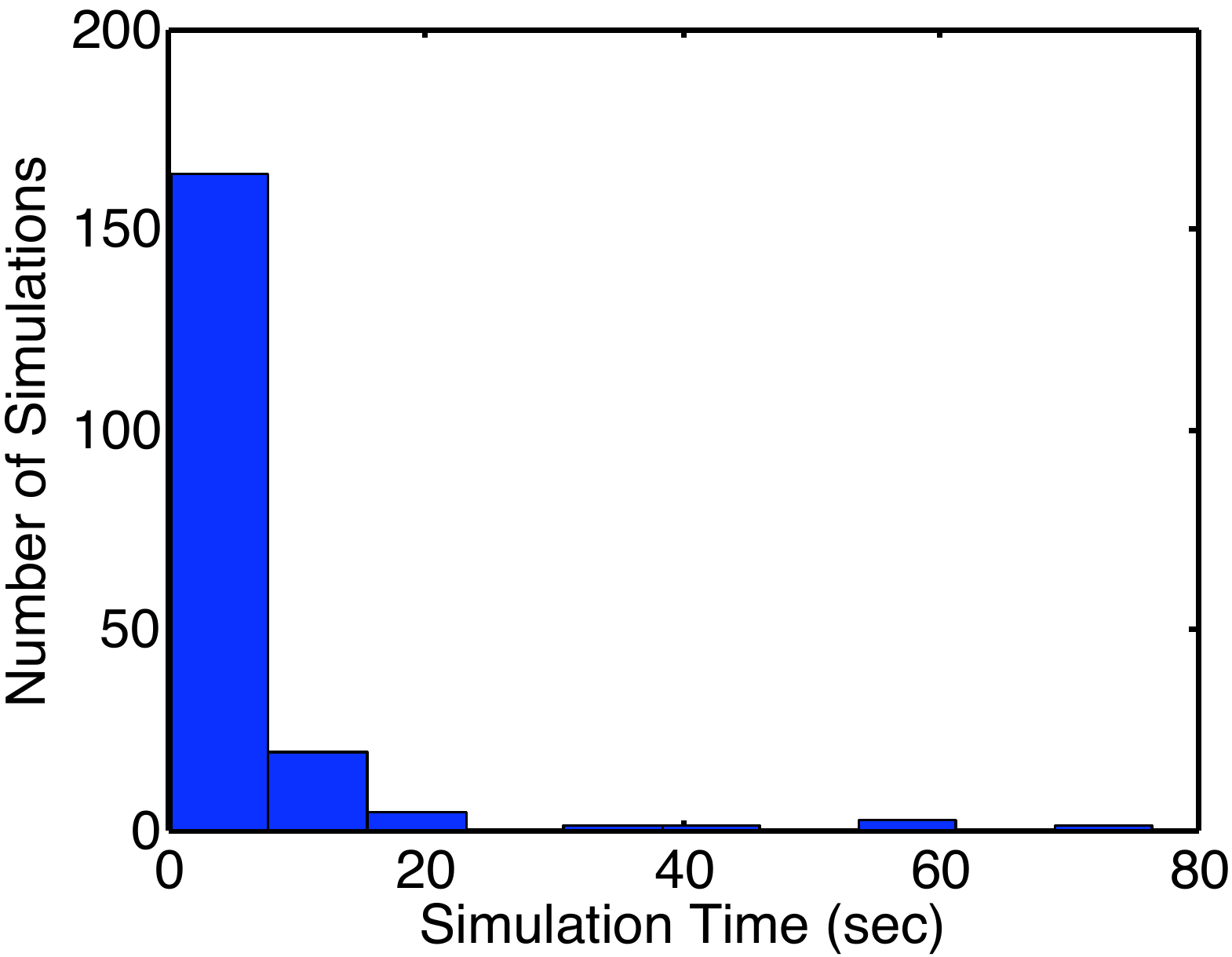}}
\caption{Safety falsification example and performance for Case II}
\label{fig:unsafety_example_hspace-unsafety_perf_hspace}
\end{figure}
\iflong \else \vspace{-0.3 in} \fi
\section{Conclusions}\label{sec:conclusions}
 In this paper, we have presented sampling-based resolution-complete algorithms for safety falsification of linear time invariant discrete-time systems over infinite time horizon. 
The algorithms attempt to
 generate a legitimate counter example by
 incrementally building feasible trajectories in the state space at
 increasing levels of resolution or
 provides a guarantee in a finite time that no such example exists,
 when the input is restricted to a certain class. As an additional result, when no
 counter example is found, the algorithms provide us with an
 arbitrarily good under approximation
 to the reachable set whose quality is independent of 
 length of trajectories.
 Efforts are currently underway to develop more efficient algorithms\iflong that combine the nice features of both the depth-first-search and the breadth-first-search strategies to explore the state space.\else . \fi
 We are also investigating if its possible to develop similar algorithms for nonlinear systems. 

\section{Acknowledgements}
We are grateful to S. LaValle, P. Cheng, S. Lindermann, and M. Branicky 
for stimulating discussion on the subject of this paper. We would also
like to thank R. Majumdar for providing constructive suggestions on issues related to 
algorithmic complexity and improvements and F. Borrelli for providing useful suggestions on implementation of multi-parametric
optimization in our work.\iflong We would also like to thank S. Lindermann for providing us with the code used to generate optimal orderings based on mutual distance, and R. Sanfelice for useful suggestions on the presentation of this paper.
\else
We would also like to thank R. Sanfelice for useful suggestions on the presentation of this paper.
\fi
The research leading to this
work was supported by the National Science Foundation (grants number 0715025 and 0325716).
\bibliography{references}

\end{document}